\documentclass{amsart}

\usepackage{graphicx}

\makeatletter
\def\subsection{\@startsection{subsection}{2}%
  \z@{.5\linespacing\@plus.7\linespacing}{.1\linespacing}%
  {\normalfont\bfseries}}
\makeatother

\makeatletter
\newcommand\mynobreakpar{\par\nobreak\@afterheading}
\makeatother

\usepackage[letterpaper,margin=1.25in]{geometry}

\usepackage{amssymb}

\usepackage[all]{xy}
\SilentMatrices %

\newtheorem{de}{Definition}[section]
\newtheorem{lem}[de]{Lemma}
\newtheorem{prop}[de]{Proposition}
\newtheorem{cor}[de]{Corollary}
\newtheorem{thm}[de]{Theorem}

\theoremstyle{remark}
\newtheorem{rem}[de]{Remark}
\newtheorem{ex}[de]{Example}

\newcommand{\blank}{-}

\newcommand{\dfn}[1]{\textbf{#1}}

\newcommand{\ra}{\to}
\newcommand{\lra}{\longrightarrow}
\DeclareMathOperator*{\colim}{colim}

\DeclareMathOperator{\stab}{stab}

\DeclareMathOperator{\id}{id}
\DeclareMathOperator{\ev}{ev}

\newcommand{\Diff}{{\mathfrak{D}\mathrm{iff}}}
\newcommand{\Top}{{\mathfrak{T}\mathrm{op}}}
\newcommand{\Set}{{\mathfrak{S}\mathrm{et}}}
\newcommand{\sSet}{{\mathfrak{s}\mathrm{Set}}}

\newcommand{\Pre}{{\mathfrak{P}\mathrm{re}}}

\newcommand{\DS}{{\mathcal{DS}}}
\newcommand{\cF}{{\mathcal{F}}}
\newcommand{\cI}{{\mathcal{I}}}
\newcommand{\cM}{{\mathcal{M}}}
\newcommand{\cN}{{\mathcal{N}}}
\newcommand{\cP}{{\mathcal{P}}}

\newcommand{\A}{\mathbb{A}}
\newcommand{\N}{\mathbb{N}}
\newcommand{\Z}{\mathbb{Z}}
\newcommand{\R}{\mathbb{R}}

\newcommand{\oo}{\infty}
\newcommand{\eps}{\epsilon}

\usepackage{ifpdf}
\ifpdf  \usepackage[pdftex,bookmarks=false]{hyperref}
\else   \usepackage[hypertex]{hyperref}
\fi

\title{The homotopy theory of diffeological spaces}
\author{J.~Daniel Christensen}
\address{Department of Mathematics, The University of Western Ontario,
London, ON N6A 5B7, Canada}
\email{jdc@uwo.ca}

\author{Enxin Wu}
\address{Faculty of Mathematics, University of Vienna,
Oskar-Morgenstern-Platz 1, 1090 Vienna, Austria}
\email{enxin.wu@univie.ac.at}

\date{May 12, 2015}

\begin{document}



\begin{abstract}
Diffeological spaces are generalizations of smooth manifolds.
In this paper, we study the homotopy theory of diffeological spaces.
We begin by proving basic properties of the smooth homotopy groups
that we will need later.
Then we introduce the smooth singular simplicial set $S^D(X)$ associated
to a diffeological space $X$, and show that when $S^D(X)$ is fibrant, it
captures smooth homotopical properties of $X$.
Motivated by this, we define $X$ to be fibrant when $S^D(X)$ is, and
more generally define cofibrations, fibrations and weak equivalences in the
category of diffeological spaces using the smooth singular functor.
We conjecture that these form a model structure, but in this
paper we assume little prior knowledge of model categories, and
instead focus on concrete questions about smooth manifolds and
diffeological spaces.
We prove that our setup generalizes the naive smooth homotopy theory of smooth manifolds
by showing that a smooth manifold without boundary is fibrant
and that for fibrant diffeological spaces, the weak equivalences can be
detected using ordinary smooth homotopy groups.
We also show that our definition of fibrations
generalizes Iglesias-Zemmour's theory of diffeological bundles.
We prove enough of the model category axioms to show that every diffeological
space has a functorial cofibrant replacement.
We give many explicit examples of objects that are cofibrant, not cofibrant,
fibrant and not fibrant,
as well as many other examples showing the richness of the theory.
For example, we show that both the free and based loop spaces of a smooth manifold are fibrant.
One of the implicit points of this paper is that the language of model
categories is an effective way to organize homotopical thinking, even
when it is not known that all of the model category axioms are satisfied.
\end{abstract}

\maketitle

\tableofcontents

\section{Introduction}

Smooth manifolds play a central role in mathematics and its applications.
However, it has long been realized that more general spaces are needed,
such as singular spaces, loop spaces and other infinite-dimensional spaces,
poorly behaved quotient spaces, etc.
Various approaches to each of these classes of spaces are available,
but there are also frameworks that encompass all of these generalizations at once.
We will discuss diffeological spaces, which were introduced by Souriau in the 1980's~\cite{So1,So2},
and which provide a well-behaved category that contains smooth manifolds as a full subcategory.
We define diffeological spaces in Definition~\ref{de:diff-space}, and we recommend that
the reader unfamiliar with diffeological spaces turn there now to see how elementary the
definition is.

Diffeological spaces by now have a long history, of which we mention just a few examples.
Diffeological spaces were invented by Souriau in order to apply
diffeological groups to problems in mathematical physics.
Donato and Iglesias-Zemmour used diffeological spaces to study irrational tori~\cite{DI}.
Later, Iglesias-Zemmour established the theory of diffeological bundles~\cite{I1}
as a setting for their previous results.
One of the key general results is the existence of a long exact sequence of smooth homotopy groups
for a diffeological bundle.
In~\cite{Da}, Dazord used diffeological spaces to study Lie groupoids, Poisson manifolds and
symplectic diffeomorphisms.
In his thesis~\cite{Wa} (as well as in a preprint with Karshon), Watts used diffeological
spaces to study the complex of differentiable forms on symplectic quotients and orbifolds.
Orbifolds were also studied by Iglesias-Zemmour, Karshon and Zadka in~\cite{IKZ}.
Tangent spaces, tangent bundles, a geometric realization functor and a smooth singular simplicial set functor
were studied by Hector in~\cite{He}, and the latter was used to define
the smooth singular homology and cohomology of diffeological spaces as
well as a notion of Kan fibration.
Tangent spaces and tangent bundles were developed further in~\cite{CW}.
Costello and Gwilliam used diffeological vector spaces in their book~\cite[Appendix~A]{CG}
as a foundation for the homological algebra of the infinite-dimensional vector spaces
that arise in their work on factorization algebras in quantum field theory,
and the general homological algebra of diffeological vector spaces was studied
in~\cite{W2}.
Iglesias-Zemmour and Karshon studied Lie subgroups of the group $\Diff(M)$ of self-diffeomorphisms
of a smooth manifold $M$ in~\cite{IK}.
Finally, the recent book~\cite{I2} by Iglesias-Zemmour provides an in-depth treatment of
diffeological spaces.

Motivated by this past work, as well as the long history of using smooth stuctures to
study homotopy theory via differential topology,
we set up a framework for the study of homotopy theory in the category $\Diff$ of diffeological spaces.
This framework generalizes the smooth homotopy theory of smooth manifolds
and Iglesias-Zemmour's theory of diffeological bundles.
Although we don't know whether our definitions produce a model structure on $\Diff$,
we use the language of model categories to express our results.
We use only basic concepts from model category theory, such
as the model structure on the category $\sSet$ of simplicial sets.
Appropriate references are~\cite{GJ,Hi,Ho,Q},
the first of which is also recommended for background on simplicial sets.
Our set-up is as follows.
Let $\A^n=\{(x_0,\ldots,x_n) \in \R^{n+1} \mid \sum_{i=0}^nx_i=1\}$
be the ``non-compact $n$-simplex,'' equipped with the sub-diffeology from $\R^{n+1}$.
For a diffeological space $X$,
let $S^D_n(X)$ denote the set of smooth maps from $\A^n$ to $X$.
These naturally form a simplicial set $S^D(X)$, and one of our main
results is that when this simplicial set is fibrant,
it captures smooth homotopical information about $X$.
More precisely, in Theorem~\ref{fib-hg-compare}
we show that when $S^D(X)$ is fibrant, the simplicial homotopy groups of $S^D(X)$
agree with the smooth homotopy groups of $X$.
This raises the question of when $S^D(X)$ is fibrant, and we answer
this question in many cases.

To organize our study of this question, we make the following definitions.
We define a map in $\Diff$ to be a weak equivalence (fibration)
if the functor $S^D$ sends it to a weak equivalence (fibration)
in the standard model structure on $\sSet$.
Cofibrations in $\Diff$ are defined by the left lifting property
(Definition~\ref{LLP-RLP}). %
As a special case of these definitions,
a diffeological space $X$ is fibrant if its smooth singular simplicial
set $S^D(X)$ is fibrant (i.e., a Kan complex).  This is a concrete condition which
says that, for each $n$, every smooth map defined on $n$ faces of $\A^n$
and taking values in $X$
extends to all of $\A^n$.
We prove that many diffeological spaces are fibrant, and in particular
that every smooth manifold $M$ without boundary is fibrant.
This is a statement that can be made without the theory of
diffeological spaces, but our proof,
which makes use of the diffeomorphism group of $M$,
illustrates the usefulness of working
with more general spaces even when studying smooth manifolds.

Our definitions are also motivated by past work, in particular the work
on irrational tori and diffeological bundles.
Irrational tori are an important test case, since they are precisely
the sort of objects that are difficult to study using traditional methods.
Donato and Iglesias-Zemmour proved that the smooth fundamental group
of an irrational torus is non-trivial,
which contrasts with the fact that
the usual (continuous) fundamental group of an irrational torus is trivial.
We prove in this paper that every irrational torus is fibrant, and
thus is a homotopically well-behaved diffeological space
that can be studied using its smooth singular simplicial set.
We also show that every diffeological bundle with fibrant fiber is a fibration,
and so we can recover Iglesais-Zemmour's theory of diffeological bundles from our work.
Another of our results shows that both the free and based loop spaces of a smooth manifold are fibrant.

We conjecture that with our definitions, $\Diff$ is a model category,%
\footnote{A preprint of Haraguchi and Shimakawa~\cite{HS} claims
to construct a model structure on the category of diffeological spaces.
Unfortunately, we have found errors in all versions of that paper.
}
and that for every diffeological space, its smooth homotopy groups
coincide with the simplicial homotopy groups of its smooth singular simplicial set.
However, our results are of interest whether or not these conjectures are true,
as the smooth singular simplicial set is a basic object of study.

\medskip

Here is an outline of the paper.

In Section~\ref{se:basics}, we review the basics of diffeological spaces,
including the facts that
the category of diffeological spaces is complete, cocomplete and cartesian closed,
and contains the category of smooth manifolds as a full subcategory.
We also discuss diffeological groups in this section.

In Section~\ref{se:pi_n}, we review the $D$-topology and smooth homotopy groups of a diffeological space
together with the theory of diffeological bundles.
This section also contains new results.
The most important result gives
many equivalent characterizations of smooth homotopy groups
using $\R^n$, $I^n$, $D^n$, spheres and simplices,
with and without stationarity conditions (Theorem~\ref{equidef}).
This result is of general interest, and is also needed in the next section.
We also compare the smooth homotopy groups of some diffeological spaces
with the usual (continuous) homotopy groups of the underlying topological spaces
(Propositions~\ref{mfd} and~\ref{CX}, and Examples~\ref{earring} and~\ref{irrtorus}),
and show that the smooth approximation theorem does not hold for general diffeological spaces (Remark~\ref{rem:sat}).

In Section~\ref{se:model}, we define the smooth singular functor, study
the extent to which this functor preserves smooth homotopical information,
and explore the basic properties of fibrant and cofibrant diffeological spaces.
In more detail, in Subsection~\ref{ss:adjoint-pair},
we use the non-compact simplices $\A^n$ in $\Diff$ to define an adjoint pair between
the category of simplicial sets and the category of diffeological spaces (Definition~\ref{adjoint}).
The functors are called the diffeological realization functor and the smooth singular functor.
We use the smooth singular functor
to define weak equivalences, fibrations and cofibrations in $\Diff$ (Definition~\ref{wfc}).
We then show that the smooth singular functor sends smoothly homotopic maps
to simplicially homotopic maps (Lemma~\ref{shtvssimhtp}),
and use this result to show that
the smooth homotopy groups of a fibrant diffeological space
and the simplicial homotopy groups of its smooth singular simplicial set agree (Theorem~\ref{fib-hg-compare}).
We also point out that
the diffeological realization functor does not commute with finite products (Proposition~\ref{noncomm}),
and compare the three adjoint pairs among $\sSet$,
$\Diff$ and $\Top$ (Propositions~\ref{3adjoint1} and~\ref{3adjoint2}).

In Subsection~\ref{ss:cofibrant}, we study cofibrations and cofibrant diffeological spaces.
We begin by proving one of the model category factorization axioms, namely that every smooth map factors into a cofibration
followed by a trivial fibration (Proposition~\ref{factorization}),
and hence that every diffeological space has a functorial cofibrant replacement (Corollary~\ref{cofrep}).
We then show that fine diffeological vector spaces and
$S^1$ are cofibrant (Propositions~\ref{pr:fine} and~\ref{S1cof}).

In Subsection~\ref{ss:fibrant}, we focus on fibrations and fibrant diffeological spaces.
We first connect our definitions to earlier work by showing that
every diffeological bundle with fibrant fiber is a fibration (Proposition~\ref{diffbundfibration}),
and then use this to show
that not every diffeological space is cofibrant (Example~\ref{noncof}).
We next show that
every diffeological group is fibrant (Proposition~\ref{diffgrpfib}),
every homogeneous diffeological space is fibrant (Theorem~\ref{homogeneousfibrant}),
and hence that every smooth manifold is fibrant (Corollary~\ref{mfdfibrant}).
In addition, we show that
every $D$-open subset of a fibrant diffeological space with the sub-diffeology is fibrant (Theorem~\ref{Dopnfibrant})
and that the function space from a (pointed) compact diffeological space to a (pointed) smooth manifold is fibrant 
(Corollary~\ref{cpt}).
This gives a second proof that every smooth manifold is fibrant,
and also shows that the free and based loop spaces of a smooth manifold are fibrant.
Finally, we show that
not every diffeological space is fibrant (Examples~\ref{nonfib1}, \ref{nonfib3}, \ref{nonfib2}
and~\ref{afewhalflines}),
and, in particular, that no smooth manifold with (non-empty) boundary is fibrant (Corollary~\ref{mfdwb-nonfib}).
\medskip

Unless otherwise specified, all smooth manifolds in this paper are
assumed to be finite-dimensional, Hausdorff, second countable and without boundary.
\medskip

We would like to thank Dan Dugger for the idea for the proof of Proposition~\ref{3adjoint1},
Gord Sinnamon for the proof of Example~\ref{halfline},
Gaohong Wang for the idea for the proof of Example~\ref{nonfib3},
and the referee for many comments that led to improvements in the exposition.

\section{Background on diffeological spaces and groups}\label{se:basics}

Here is some background on diffeological spaces and diffeological groups.
While we often cite early sources, almost all of the material
in this section is in the book~\cite{I2}, which we recommend as a good reference.
For a three-page introduction to diffeological spaces, we recommend~\cite[Section~2]{CSW},
which we present in a condensed form here.

\begin{de}[\cite{So2}]\label{de:diff-space}
A \dfn{diffeological space} is a set $X$
together with a specified set $\mathcal{D}_X$ of functions $U \ra X$ (called \dfn{plots})
for each open set $U$ in $\R^n$ and for each $n \in \N$,
such that for all open subsets $U \subseteq \R^n$ and $V \subseteq \R^m$:
\begin{enumerate}
\item (Covering) Every constant map $U \ra X$ is a plot;
\item (Smooth Compatibility) If $U \ra X$ is a plot and $V \ra U$ is smooth,
then the composition $V \ra U \ra X$ is also a plot;
\item\label{de:diff-space3} (Sheaf Condition) If $U=\cup_i U_i$ is an open cover
and $U \ra X$ is a set map such that each restriction $U_i \ra X$ is a plot,
then $U \ra X$ is a plot.
\end{enumerate}
We usually use the underlying set $X$ to denote the diffeological space $(X,\mathcal{D}_X)$.
\end{de}

\begin{de}[\cite{So2}]
Let $X$ and $Y$ be two diffeological spaces,
and let $f:X \rightarrow Y$ be a set map.
We say that $f$ is \dfn{smooth} if for every plot $p:U \ra X$ of $X$,
the composition $f \circ p$ is a plot of $Y$.
\end{de}

The collection of all diffeological spaces and smooth maps forms a category,
which we denote $\Diff$.
Given two diffeological spaces $X$ and $Y$,
we write $C^\infty(X,Y)$ for the set of all smooth maps from $X$ to $Y$.
An isomorphism in $\Diff$ will be called a \dfn{diffeomorphism}.

Every smooth manifold $M$ is canonically a diffeological space
with the same underlying set and plots taken to be all smooth maps $U \ra M$ in the usual sense.
We call this the \dfn{standard diffeology} on $M$, and, unless we say
otherwise, we always equip a smooth manifold with this diffeology.
By using charts, it is easy to see that smooth maps in the usual sense between smooth manifolds
coincide with smooth maps between them with the standard diffeology.

The smallest diffeology on $X$ containing a set of
maps $A=\{U_i \ra X\}_{i \in I}$ is called the diffeology \dfn{generated} by $A$.
It consists of all maps $U \ra X$ that locally either factor through the given
maps via smooth maps, or are constant.

For a diffeological space $X$ with an equivalence relation~$\sim$,
the smallest diffeology on $X/{\sim}$ making the quotient map $\{ X \twoheadrightarrow X/{\sim} \}$ smooth
is called the \dfn{quotient diffeology}.
It consists of all maps $U \ra X/{\sim}$ that locally factor through the quotient map.
For a diffeological space $X$ and a subset $A$ of $X$,
the largest diffeology on $A$ making the inclusion map $\{A \hookrightarrow X\}$ smooth
is called the \dfn{sub-diffeology}.
It consists of all maps $U \ra A$ such that $U \ra A \hookrightarrow X$ is a plot of $X$.

\begin{thm}
The category $\Diff$ is both complete and cocomplete.
\end{thm}

This is proved in~\cite{BH}, but can be found implicitly in earlier work.
A concise sketch of a proof is provided in~\cite{CSW}.
As a special case, the set-theoretic cartesian product $X \times Y$ of
two diffeological spaces has a diffeology consisting of those functions
$U \to X \times Y$ such that the component functions $U \ra X$ and $U \ra Y$
are plots, and this diffeology makes $X \times Y$ into the product of
$X$ and $Y$ in $\Diff$.

The category of diffeological spaces also enjoys another convenient property.
Given two diffeological spaces $X$ and $Y$,
the set of functions $\{U \ra C^\infty(X,Y) \mid U \times X \ra Y \text{ is smooth} \}$
forms a diffeology on $C^\infty(X,Y)$.
We always equip hom-sets with this diffeology.
One can show~\cite{I1, BH, I2} that
for each diffeological space $Y$,
$\blank \times Y:\Diff \rightleftharpoons \Diff :C^\infty(Y,\blank)$ is an adjoint pair,
which gives the following result:

\begin{thm}\label{th:cartesian-closed}
The category $\Diff$ is cartesian closed.
\end{thm}

\begin{rem}\label{rem:concrete}
We now describe an alternate point of view that is not needed for any
of the results of the paper, but which explains the good categorical
properties of $\Diff$.
Write $\DS$ for the category with objects the open subsets of
$\R^n$ for all $n \in \N$ and morphisms the smooth maps between them.
We associate to each diffeological space $X$ a presheaf $\cF_X$ on $\DS$
which sends an open subset $U$ of some $\R^n$ to the set $\cF_X(U)$ of all plots
from $U$ to $X$.
This is a sheaf with respect to the notion of cover described in
Definition~\ref{de:diff-space}(\ref{de:diff-space3}).
Any (pre)sheaf $\cP$ on $\DS$ has a natural map
\[
  \cP(U) \lra \Set(U, \cP(\R^0))
\]
sending an element $s$ to the function $U \to \cP(\R^0)$ which
sends $u$ to $i_u^*(s)$, where $i_u$ is the map $\R^0 \to U$ sending the one
point space $\R^0$ to $u$.
When this map is injective for every $U$, $\cP$ is said to be a \dfn{concreate (pre)sheaf}.
It is easy to show that every $\cF_X$ is concrete, and moreover
that the category $\Diff$ is equivalent to the category of concrete sheaves on $\DS$,
a full subcategory of the category of sheaves
(see~\cite[Proposition~4.15]{BH} and~\cite{Du}).
The category of concrete sheaves on any concrete site is a
``Grothendieck quasi-topos,'' and is always complete, cocomplete,
locally cartesian closed and locally presentable~\cite[Theorem~C2.2.13]{J}.
\end{rem}

Finally, we discuss diffeological groups, which will be useful in Subsections~\ref{ss:bundles}
and~\ref{ss:fibrant}.

\begin{de}[\cite{So1}]
A \dfn{diffeological group} is a group object in $\Diff$.
That is, a diffeological group is both a diffeological space and a group
such that the group operations are smooth maps.
\end{de}

\begin{ex}\
\begin{enumerate}
\item Any Lie group with its standard diffeology is a diffeological group.

\item The continuous diffeological space $C(G)$ of a topological group $G$ is a diffeological group,
where the functor $C$ is defined just before Proposition~\ref{Difftopadj}.
\end{enumerate}
\end{ex}

\begin{ex}
Let $G$ be a diffeological group,
and let $H$ be any subgroup of $G$.
Then $H$ with the sub-diffeology is automatically a diffeological group.
If $H$ is a normal subgroup of $G$,
then the quotient group $G/H$ with the quotient diffeology is also a diffeological group.
\end{ex}

\begin{ex}[\cite{So1}]
Let $X$ be a diffeological space.
Write $\Diff(X)$ for the set of all diffeomorphisms from $X$ to itself.
Then $\{ p: U \ra \Diff(X) \mid \text{the two maps } U \times X \ra X \text{ defined by }
(u,x) \mapsto p(u)(x) \text{ and } (u,x) \mapsto (p(u))^{-1}(x) \text{ are smooth} \}$
is a diffeology on $\Diff(X)$,
which makes $\Diff(X)$ into a diffeological group.
When $X$ is a smooth manifold with the standard diffeology,
the above diffeology on $\Diff(X)$ is in fact the sub-diffeology from $C^\oo(X, X)$.
\end{ex}

\section{Smooth homotopy groups and diffeological bundles}\label{se:pi_n}

In this section, we begin by setting up the basics of the smooth homotopy theory of diffeological spaces
and giving several equivalent characterizations of smooth homotopy groups.
We then review the $D$-topology and diffeological bundles from~\cite{I1}.
We show that the smooth approximation theorem does not hold for general diffeological spaces (see Remark~\ref{rem:sat}),
and that the smooth homotopy groups of a diffeological space
and the usual (continuous) homotopy groups of the underlying topological space
do not match in general (see Example~\ref{irrtorus}).

\medskip

Throughout the paper, we will make use of the following definition.

\begin{de}
For $0 < \eps < 1/2$, an \dfn{$\eps$-cut-off function} is a smooth function
$\phi : \R \to \R$ such that $0 \leq \phi(t) \leq 1$ for $t \in \R$,
$\phi(t) = 0$ if $t < \eps$ and $\phi(t) = 1$ if $t > 1 - \eps$.
A \dfn{cut-off function} is an $\eps$-cut-off function for some $0 < \eps < 1/2$.
It is well-known that such functions exist for all such $\eps$.
\end{de}

\subsection{Smooth homotopy groups}

We begin with the elementary smooth homotopy theory of diffeological spaces,
leading up to Iglesias-Zemmour's recursive definition
of the smooth homotopy groups of a diffeological space $X$; see~\cite{I1}.
The main result of this subsection is Theorem~\ref{equidef}, which shows
that many definitions of smooth homotopy groups agree.

A \dfn{path} in $X$ is a smooth map $f : \R \to X$.
We say that $f$ is \dfn{stationary} if there is an $\eps > 0$ such that
$f$ is constant on $(-\oo, \eps)$ and also on $(1-\eps, \oo)$.

We define a relation on $X$ by $x \simeq y$
if and only if there is a smooth path $f$ connecting $x$ and $y$,
that is, with $f(0) = x$ and $f(1) = y$.
When this is the case, the path can always be chosen to be stationary,
because of the existence of cut-off functions.
It follows that $\simeq$ is an equivalence relation,
and that $x \simeq y$ if and only if there is a smooth function $f : I \to X$
with $f(0) =x$ and $f(1) = y$, where $I = [0, 1] \subset \R$ has the sub-diffeology.
The equivalence classes are called the \dfn{smooth path components}, and
the \dfn{$0^{th}$ smooth homotopy group $\pi_0^D(X)$} is defined to be the quotient set $X/{\simeq}$.
As usual, for $x \in X$, $\pi_0^D(X,x)$ denotes the set $\pi_0^D(X)$ pointed by the path component of $x$.

Let $X$ and $Y$ be diffeological spaces.
We say that smooth maps $f, g: X \ra Y$ are \dfn{smoothly homotopic}
if $f \simeq g$ as elements of $C^{\oo}(X,Y)$.
By cartesian closedness of $\Diff$,
$f \simeq g$ if and only if there exists a smooth map $F : X \times \R \to Y$ (or $F: X \times I \to Y$)
such that $F(x,0)=f(x)$ and $F(x,1)=g(x)$
for each $x$ in $X$.
It is easy to see that
smooth homotopy is an equivalence relation compatible with both left and right composition.
We write $[X,Y]$ for $\pi_0^D(C^\infty(X,Y))$, the set of smooth homotopy classes.

A \dfn{pair} is a diffeological space $X$ with a chosen diffeological subspace $U$.
A map $f : (X, U) \to (Y, V)$ of pairs is a smooth map $f : X \to Y$
such that $f(U) \subseteq V$, and the function space $C^{\oo}((X,U), (Y,V))$
is the set of such maps with the sub-diffeology from $C^{\oo}(X, Y)$.
Two such maps are \dfn{smoothly homotopic} if they are in the same path
component of $C^{\oo}((X,U), (Y,V))$,
and we write $[(X,U),(Y,V)]$ for $\pi_0^D(C^\infty((X,U),(Y,V)))$.
Two pairs $(X,U)$ and $(Y,V)$ are \dfn{smoothly homotopy equivalent}
if there are maps $f : (X,U) \to (Y,V)$ and $g: (Y,V) \to (X,U)$ such
that $fg$ and $gf$ are smoothly homotopic to the identity maps.
When $U$ consists of a single point $x$, we write $(X,x)$ for $(X,\{x\})$
and call this a \dfn{pointed diffeological space}.

Now let $(X,x)$ be a pointed diffeological space.
The \dfn{loop space} of $(X,x)$ is the space
$\Omega(X,x) = C^{\oo}((\R, \{0, 1\}), (X,x))$,
with basepoint the constant loop at $x$.
We inductively define $\Omega^0(X,x) = (X,x)$ and,
for $n > 0$, $\Omega^n(X,x) = \Omega(\Omega^{n-1}(X,x))$.
The \dfn{$n^{th}$ smooth homotopy group $\pi_n^D(X,x)$} is defined to be $\pi_0^D(\Omega^n(X,x))$.
For $n \geq 1$, $\pi_n^D(X,x)$ is a group:  the product is defined by observing that each loop
is equivalent to a stationary loop and composing such loops in the usual way.
One can show that $\pi_n^D(X,x)$ is an abelian group if $n \geq 2$.
These constructions are functorial.

To avoid needing to choose stationary loops for the group multiplication,
one can require all paths and loops appearing above to be stationary.
This gives the stationary loop spaces $\tilde{\Omega}^n(X,x)$ and
new functors $\tilde{\pi}_n^D$.  It is not hard to show that
there is a natural isomorphism $\tilde{\pi}_n^D(X,x) \cong \pi_n^D(X,x)$ for each $n \geq 0$.

Since $\Diff$ is cartesian closed (see Theorem~\ref{th:cartesian-closed}),
a function $f$ in $\Omega^n(X,x)$ can be regarded as a smooth map
$\tilde{f} : \R^n \to X$ which sends
$\{(x_1,\ldots,x_n) \in \R^n \mid x_i=0 \text{ or } 1 \text{ for some } i \}$ to $x$,
and $\pi_n^D(X, x)$ consists of smooth path components in the space of such maps.
Unfortunately, while the stationarity condition makes composition easier,
if the definition of $\tilde{\pi}_n^D(X, x)$ is unravelled it leads to
a highly irregular condition on maps $\tilde{f} : \R^n \to X$ because
the $\epsilon$ can vary in an uncontrolled way.

We next show that a variety of natural definitions of the smooth
homotopy groups of a pointed diffeological space agree.
This will be used to prove Theorem~\ref{fib-hg-compare}.

\begin{thm}\label{equidef}
For each pair $(A,B)$ of diffeological spaces listed below,
there is a natural bijection between $[(A,B), (X,x)]$ and $\pi_n^D(X,x)$,
where $(X,x)$ is a pointed diffeological space.
\begin{enumerate}
\item\label{item:1} $(\R^n,\partial \R^n)$,
where $\partial \R^n=\{(x_1,\ldots,x_n) \in \R^n \mid x_i=0 \text{ or } 1 \text{ for some } i \}$;

\item\label{item:2} $(\R^n,\partial_{\eps} \R^n)$,
where $\partial_{\eps} \R^n = \{(x_1,\ldots,x_n) \in \R^n
\mid x_i \leq \eps \text{ or } x_i \geq 1-\eps \text{ for some } i \}$,
and $\epsilon \in (0, 1/2)$ is a fixed real number;

\item\label{item:3} $(I^n,\partial I^n)$,
where $I^n$ is the unit cube in $\R^n$ with the sub-diffeology,
and $\partial I^n$ is its boundary;

\item\label{item:4} $(I^n, \partial_{\eps} I^n)$,
where
$\partial_{\eps} I^n = \{(x_1,\ldots,x_n) \in I^n
\mid x_i \leq \eps \text{ or } x_i \geq 1-\eps \text{ for some } i\}$,
and $\epsilon \in (0, 1/2)$ is a fixed real number;

\item\label{item:5} $(\mathbb{A}^n,\partial \mathbb{A}^n)$,
where $\mathbb{A}^n=\{(x_0,x_1,\ldots,x_n) \in \R^{n+1} \mid \sum x_i = 1 \}$ with the sub-diffeology, and $\partial \mathbb{A}^n=\{(x_0,\ldots,x_n) \in \mathbb{A}^n \mid x_i=0 \text{ for some } i \}$;

\item\label{item:6} $(\mathbb{A}^n,\partial_{\eps} \A^n)$,
where $\partial_{\eps} \A^n = \{(x_0,\ldots,x_n) \in \mathbb{A}^n \mid x_i \leq \epsilon \text{ for some } i \}$,
and $\epsilon \in (0, 1/(n+1))$ is a fixed real number;

\item\label{item:7} $(D^n, \partial D^n)$, where $D^n$ is the unit ball in $\R^n$, and
$\partial D^n = S^{n-1}$ is the unit sphere;

\item\label{item:8} $(D^n, \partial_{\eps} D^n)$,
where $\partial_{\eps} D^n = \{ x \in D^n \mid \|x\| > 1 - \eps \}$,
and $\epsilon \in (0, 1/2)$ is a fixed real number;

\item\label{item:9} $(S^n,N)$,
where $S^n$ is the unit sphere in $\R^{n+1}$, and $N=(0,\ldots,0,1)$ is the north pole.

\end{enumerate}
In fact, in cases (\ref{item:1}) through (\ref{item:8}), the pairs are smoothly homotopy equivalent.
\end{thm}

\begin{proof}
As explained earlier, it follows from the cartesian closedness of
$\Diff$ that $[(\R^n, \partial\R^n), (X,x)] \cong \pi_n^D(X,x)$.

One can also show, using cartesian closedness several times, that
if pairs $(A,B)$ and $(A',B')$ are smoothly homotopy equivalent,
then so are the diffeological spaces $C^\oo((A,B),(X,U))$ and
$C^\oo((A',B'),(X,U))$ for any pair $(X,U)$.
In particular, $[(A,B),(X,U)] \cong [(A',B'),(X,U)]$.
We will prove that the pairs (\ref{item:1}) through (\ref{item:8}) are smoothly homotopy
equivalent, and then separately prove that the pair (\ref{item:9}) gives rise
to an equivalent set of homotopy classes.

(\ref{item:1}) $\iff$ (\ref{item:2}):
Consider the inclusion map $(\R^n, \partial \R^n) \hookrightarrow (\R^n, \partial_{\eps} \R^n)$
and the map $\phi^n \!: (\R^n, \partial_{\eps} \R^n) \to (\R^n, \partial \R^n)$
which applies an $\eps$-cut-off function coordinate-wise.
Both composites are homotopic to the identity via the affine homotopy,
which can be checked to preserve the appropriate subsets.

(\ref{item:3}) $\iff$ (\ref{item:2}) and (\ref{item:4}) $\iff$ (\ref{item:2}): These are proved using the same
argument, by considering the inclusion into $(\R^n, \partial_{\eps} \R^n)$
and the map $\phi^n$ in the other direction.

(\ref{item:6}) $\iff$ (\ref{item:2}):
Consider the diffeomorphism $\psi: \A^n \ra \R^n$ defined by
$(x_0,x_1,\ldots,x_n) \ra (x_1,\ldots,x_n)$.
Its inverse sends $\partial_{\eps} \R^n$ into $\partial_{\eps} \A^n$,
but $\psi$ itself is not a map of pairs.
However, if we first dilate $\A^n$ by a large enough factor
and then apply $\psi$, this is a map of pairs,
and it is easy to see that the two composites are smoothly homotopic to the identity.

(\ref{item:5}) $\iff$ (\ref{item:6}): Consider the inclusion
$i : (\mathbb{A}^n,\partial \mathbb{A}^n) \hookrightarrow
     (\mathbb{A}^n,\partial_{\eps} \A^n)$.
We construct a map in the other direction as follows.
Let $\rho: \R \ra \R$ be an $\eps$-cut-off function
such that $\rho(y) > 0$ if $y \geq \frac{1}{n+1}$.
Define
\[
  u: \R^{n+1} \setminus \{(x_0,\ldots,x_n) \in \R^{n+1} \mid \sum x_i = 0 \} \lra \mathbb{A}^n
\]
by
\[
  u(x_0,\ldots,x_n) = \left( \frac{x_0}{\sum x_i}, \ldots, \frac{x_n}{\sum x_i} \right).
\]
Then $u \circ \rho^{n+1}$ is a well-defined map from $(\mathbb{A}^n,\partial_{\eps} \A^n)$
to $(\mathbb{A}^n,\partial \mathbb{A}^n)$:
we have $\sum \rho(x_i) > 0$ since $x_i \geq 1/(n+1)$ for some $i$.
If we replace $\rho$ with the affine homotopy $\alpha_t$ defined by $\alpha_t(y) = t y + (1-t)\rho(y)$,
then $\sum \alpha_t(x_i) > 0$ for each $t \in I$,
and so $u \circ \alpha_t^{n+1}$ is well-defined and smooth as a function $\A^n \times I \to \A^n$.
It provides a smooth homotopy between $u \circ \rho^{n+1} \circ i$ and the identity
on $(\A^n, \partial \A^n)$.
Moreover, the composite $i \circ u \circ \alpha_t^{n+1}$ provides a smooth homotopy
between $i \circ u \circ \rho^{n+1}$ and the identity on $(\A^n, \partial_{\eps} \A^n)$.

(\ref{item:8}) $\iff$ (\ref{item:2}): This is similar to (\ref{item:6}) $\iff$ (\ref{item:2});  each pair includes in the
other after an appropriate scaling.

(\ref{item:7}) $\iff$ (\ref{item:8}): This is proved using a radial cut-off function.

(\ref{item:9}):
Write $H$ for the northern hemisphere of $S^n$, and recall that $N$ is the north pole.
Using stereographic projection and suitable rescalings, one can see
using the methods above that the pairs $(S^n \setminus N, H \setminus N)$
and $(\R^n, \partial_{\eps} \R^n)$ are smoothly homotopy equivalent.
Next, observe that the mapping spaces
$C^\oo((S^n \setminus N, H \setminus N), (X,x))$ and
$C^\oo((S^n, H), (X,x))$ are diffeomorphic, since every constant function
on $H \setminus N$ extends uniquely to a smooth function on $H$.
Finally, by gradually raising the equator and using a cut-off function,
one sees that the pairs $(S^n, H)$ and $(S^n, N)$ are smoothly homotopy equivalent.
\end{proof}

Theorem~\ref{equidef} implies that for each of the pairs $(A, B)$ considered
above, the set $[(A,B), (X,x)]$
inherits a natural group structure from $\pi_n^D(X,x)$, for $n \geq 1$.
One can make the formulas explicit as needed, by using the maps between
the pairs that were described in the proof.

\begin{rem}
Similar methods apply to other pairs, which impose variants on the
above stationarity conditions.
For example, one can consider the pairs (\ref{item:2}), (\ref{item:4}), (\ref{item:6}) and (\ref{item:8}) with $\eps =  0$.
While the proofs above do not go through in this case, it is nevertheless
easy to see, for example, that the pairs $(\R^n, \partial_0 \R^n)$ and
$(\R^n, \partial_{\eps} \R^n)$ are smoothly homotopy equivalent for $\eps > 0$,
and so $(\R^n, \partial_0 \R^n)$ may also be used in the definition of
the smooth homotopy groups.

In another direction, one can also show that it is equivalent to allow
$\eps$ to vary.  For example, in case (\ref{item:2}), one could consider the
set $\{ f  \in C^\oo(\R^n, X) \mid f(\partial_{\eps} \R^n) = x \text{ for some } \eps > 0 \}$
with the sub-diffeology from $C^\oo(\R^n, X)$.
Similar methods show that the set of path components of this space again
bijects with $\pi_n^D(X,x)$.
\end{rem}

\medskip

The following result is straightforward.

\begin{prop}
Let $\{(X_j,x_j)\}_{j \in J}$ be a family of pointed diffeological spaces.
Then the canonical map
$\pi_n^D(\prod_j X_j, (x_j)) \ra \prod_j \pi_n^D(X_j, x_j)$
is an isomorphism, for each $n \in \N$.
\end{prop}

\begin{prop}
If $G$ is a diffeological group with identity $e$,
then $\pi_0^D(G)$ is a group, and $\pi_1^D(G,e)$ is an abelian group.
\end{prop}
\begin{proof}
This is formal.
\end{proof}

\subsection{The $D$-topology}\label{ss:D-topology}

In this subsection we recall the $D$-topology, which is a natural
topology on the underlying set of any diffeological space.
We summarize the basic properties of the $D$-topology,
and compare the smooth homotopy groups of a diffeological space
and the usual (continuous) homotopy groups of its underlying topological space.
Finally, we observe that the smooth approximation theorem does not hold for general diffeological spaces.

\begin{de}[\cite{I1}]
Given a diffeological space $X$, the final topology induced by its plots,
where each domain is equipped with the standard topology,
is called the \dfn{$D$-topology} on $X$.
\end{de}

\begin{ex}
The $D$-topology on a smooth manifold with the standard diffeology
coincides with the usual topology on the manifold.
\end{ex}

A smooth map $X \ra X'$ is continuous when $X$ and $X'$ are equipped with the
$D$-topology, and so this defines a functor $D: \Diff \ra \Top$ to the
category of topological spaces.

Every topological space $Y$ has a natural diffeology,
called the \dfn{continuous diffeology}, whose plots $U \ra Y$
are the continuous maps.
A continuous map $Y \ra Y'$ is smooth when $Y$ and $Y'$ are equipped
with the continous diffeology, and so this defines a functor
$C: \Top \ra \Diff$.

\begin{prop}[\cite{SYH}]\label{Difftopadj}
The functors $D: \Diff \rightleftharpoons \Top :C$ are adjoint.
\end{prop}

For more discussion on the $D$-topology, see~\cite[Chapter 2]{I2} and~\cite{CSW}.
In the rest of this subsection,
we focus on the comparison between the smooth homotopy groups of a diffeological space $X$
and the usual (continuous) homotopy groups of $D(X)$.

\medskip

By Theorem~\ref{equidef}, for any $n \in \N$, there is a natural transformation $j_n:\pi_n^D(X,x) \ra \pi_n (D(X), x)$.

\begin{prop}[\cite{I1}]
Let $X$ be a diffeological space.
Then $j_0:\pi_0^D(X) \ra \pi_0(D(X))$ is a bijection. That is,
$\pi_0^D(X)$ coincides with the usual (continuous) path components of $X$ under the $D$-topology.
\end{prop}

The classical smooth approximation theorem shows:

\begin{prop}\label{mfd}
Let $(X,x)$ be a pointed smooth manifold.
Then $j_n:\pi_n^D(X,x) \ra \pi_n(D(X),x)$ is an isomorphism for each $n \in \N$.
\end{prop}

The following result is easy to prove.

\begin{prop}\label{CX}
For any pointed topological space $(X,x)$,
the canonical map $\pi_n^D(C(X),x) \ra \pi_n(X,x)$ is an isomorphism for each $n \in \N$.
\end{prop}

In general, $j_n$ may fail to be injective (Example~\ref{irrtorus}) or surjective (Example~\ref{earring}).
In fact, there may be no isomorphism between $\pi_n^D(X,x)$ and $\pi_n(D(X),x)$ (Example~\ref{irrtorus}).

\begin{ex}[Hawaiian earring]\label{earring}
Let $X = \cup_{n=1}^\infty \{ (a,b) \in \R^2 \mid (a-1/n)^2+b^2=1/n^2 \}$,
the union of circles of radius $1/n$ and center $(1/n, 0)$,
with the sub-diffeology from $\R^2$, and let $x=(0,0) \in X$.
We will show that the map $j_1:\pi_1^D(X,x) \ra \pi_1(D(X),x)$ is not surjective.
First we show that the $D$-topology on $X$ is the same as the sub-topology of $\R^2$.
It is enough to show that each $D$-open neighbourhood of $x$ contains all but finitely many circles.
So suppose $A$ is a subset containing $x$ but not containing infinitely many circles,
and choose a sequence $x_1, x_2, \ldots$ on circles of decreasing radii but not in $A$.
If the circles are chosen so that the radii decrease sufficiently quickly,
then there is a smooth curve $p : \R \to X$ which passes through these points in order,
say $p(t_i) = x_i$ with $0 < t_1 < t_2 < \cdots < 1$.
Then the $t_i$'s are not in $p^{-1}(A)$ but their limit $t$ is, since
we must have $p(t) = \lim x_i = x$.
Thus $A$ is not $D$-open.
Therefore, the $D$-topology on $X$ is the same as the sub-topology of $\R^2$.

Now we show that the map $j_1:\pi_1^D(X,x) \ra \pi_1(D(X),x)$ is not surjective.
This is because there is no smooth stationary curve $\R \ra X$ going around
every circle in $X$,
since the sum of the circumferences of all these circles is infinite,
and any smooth curve defined on a compact interval can only travel a finite distance.
Here we are using the fact the $\pi_1(D(X), x)$ contains an element that is
not represented by a path that is not surjective;
see, for example,~\cite[Section~71, Example~1]{M}.
\end{ex}

\begin{rem}\label{rem:sat}
The above example shows that the smooth approximation theorem does not hold for a general diffeological space $X$,
in the sense that if $f:S^n \ra D(X)$ is a continuous map,
then there may not exist a smooth map $g:S^n \ra X$ such that $f$ is (continuously) homotopic to $D(g)$.
\end{rem}

\subsection{Diffeological bundles}\label{ss:bundles}

Diffeological bundles are analogous to fiber bundles, but
are much more general than the most obvious notion of locally trivial bundle.
We review diffeological bundles in this subsection,
and reach the conclusion that the smooth homotopy groups of a diffeological space
are in general different from the usual (continuous) homotopy groups of its underlying topological space.
All material here can be found in~\cite{I1} and~\cite{I2}.

\begin{de}
Let $F$ be a diffeological space.
A smooth map $f:X \rightarrow Y$ between two diffeological spaces
is \dfn{trivial of fiber type $F$}
if there exists a diffeomorphism $h:X \rightarrow F \times Y$,
where $F \times Y$ is equipped with the product diffeology,
such that the following diagram is commutative:
\[
\xymatrix@C5pt{X \ar[dr]_f \ar[rr]^-h && F \times Y\  \ar[dl]^{pr_2} \\ & Y.}
\]
The map $f$ is \dfn{locally trivial of fiber type $F$}
if there exists a $D$-open cover $\{ U_i \}$ of $Y$ such that
$f|_{f^{-1}(U_i)}:f^{-1}(U_i) \ra U_i$ is trivial of fiber type $F$ for each $i$.
\end{de}

Being locally trivial turns out to be too strong a condition for many
applications, but is the correct notion for open subsets of $\R^n$.

\begin{de}
A smooth surjective map $f:X \ra Y$ between two diffeological spaces
is called a \dfn{diffeological bundle of fiber type $F$} if the pullback of $f$ along any plot
of $Y$ is locally trivial of fiber type $F$.
In this case, we call $F$ the \dfn{fiber} of $f$, $X$ the \dfn{total space}, and $Y$ the \dfn{base space}.
\end{de}

In~\cite{I2}, diffeological bundles are defined using groupoids,
but~\cite[8.9]{I2} shows that the definitions are equivalent.
Moreover, there is another equivalent characterization:

\begin{prop}[{\cite[8.19]{I2}}]\label{pr:global-plot}
A smooth surjective map $f:X \ra Y$ between two diffeological spaces
is a diffeological bundle of fiber type $F$ if and only if
the pullback of $f$ along any global plot of $Y$
(that is, a plot of the form $\R^n \ra Y$) is trivial of fiber type $F$.
\end{prop}

\begin{ex}
Every smooth fiber bundle over a smooth manifold is a diffeological bundle.
\end{ex}

\begin{prop}[{\cite[8.15]{I2}}]\label{pr:diffgroup-diffbundle}
Let $G$ be a diffeological group, and let $H$ be a subgroup of $G$ with the sub-diffeology.
Then $G \ra G/H$ is a diffeological bundle of fiber type $H$,
where $G/H$ is the set of left (or right) cosets of $H$ in $G$, with the quotient diffeology.
\end{prop}

Note that we are \emph{not} requiring the subgroup $H$ to be closed.

\begin{thm}[{\cite[8.21]{I2}}]\label{LES}
Let $f:X \ra Y$ be a diffeological bundle of fiber type $F=f^{-1}(y)$
(equipped with the sub-diffeology from $X$) for some $y \in Y$.
Then for any $x \in F$, we have the following long exact sequence of smooth homotopy groups:
\[
\xymatrixcolsep{0.2in}
\xymatrix{\cdots \ar[r] & \pi_n^D(F,x) \ar[r]^{i_*} & \pi_n^D(X,x) \ar[r] & \pi_n^D(Y,y) \ar[r]
& \pi_{n-1}^D(F,x) \ar[r] & \cdots \ar[r] & \pi_0^D(Y) \ar[r] & 0.}
\]
\end{thm}

\begin{ex}[{\cite[8.38]{I2}}]\label{irrtorus}
Let $T^2=\R^2/\Z^2$ be the usual $2$-torus,
and let $\R_\theta$ be the image of the line $\{ y = \theta x \}$
under the quotient map $\R^2 \ra T^2$, with $\theta$ a fixed irrational number.
Note that $T^2$ is an abelian Lie group,
and $\R_\theta$ is a subgroup which is diffeomorphic to $\R$.
The quotient group $T^2_\theta := T^2/\R_\theta$ with the quotient diffeology
is called the \dfn{irrational torus of slope $\theta$}, and
by Proposition~\ref{pr:diffgroup-diffbundle},
the quotient map $T^2 \ra T^2_\theta$ is a diffeological bundle of fiber type $\R_\theta$.
By Theorem~\ref{LES},
$\pi^D_1(T^2_\theta) \cong \pi^D_1(T^2) \cong \Z \oplus \Z$.
But as a topological space with the $D$-topology, $\pi_1(T^2_\theta) \cong 0$,
since the $D$-topology on $T^2_\theta$ is indiscrete.
This follows from the fact that the $D$-topology of a quotient diffeological space
coincides with the quotient topology of the $D$-topology of the original space,
since the functor $D:\Diff \ra \Top$ is a left adjoint.

Note also that the diffeological bundle $T^2 \ra T^2_\theta$
is not locally trivial, since this would imply that it is trivial.
However, any smooth section $T^2_\theta \ra T^2$ would be induced
by a smooth map $T^2 \to T^2$ which is constant
on the dense subspace $\R_{\theta}$.  Thus it would be constant,
and could not be a section.
\end{ex}

Note that the irrational tori are trivial in approaches
to generalizing smooth manifolds which are based on ``mapping out''
rather than ``mapping in''.
(See~\cite{St} for a comparison between different approaches.)

\begin{rem}
Example~\ref{irrtorus} shows that the smooth homotopy groups of a diffeological space $X$
have more information than the usual (continuous) homotopy groups of $D(X)$, and more
generally that $X$ contains more information than $D(X)$.
We would like our homotopy theory to encode this information.
Hence, we will not use the functor $D:\Diff \ra \Top$ to define weak
equivalences in $\Diff$.
Instead, we will define an adjoint pair $|\blank|_D:\sSet \rightleftharpoons \Diff:S^D$
in the coming section, and we will use the functor $S^D$ to define the
weak equivalences in $\Diff$, as we will show that in good cases it
retains the information about the smooth homotopy groups.
\end{rem}

\section{The homotopy theory of diffeological spaces}\label{se:model}

In this section, we define the smooth singular simplicial set $S^D(X)$
associated to a diffeological space $X$, and also study the diffeological
realization functor which is left adjoint to $S^D$.
It is well-known that the singular simplicial set associated to a
topological space captures homotopical information about the space,
and one of our main results is that the same is true in the diffeological
setting, when $S^D(X)$ is a fibrant simplicial set.
Motivated by this, we define a diffeological space $X$ to be fibrant
when $S^D(X)$ is fibrant, and more generally define fibrations, cofibrations
and weak equivalences of diffeological spaces using this adjoint pair.
Although we don't know whether the definitions we give satisfy the axioms of a model category,
we prove that a wide variety of diffeological spaces, including smooth manifolds, are fibrant,
which shows that the above result is broadly applicable.
We also prove that our fibrations are closely related to diffeological bundles,
which shows that our definitions recover the usual smooth homotopy theory of
smooth manifolds as well as past work on the smooth homotopy theory of diffeological bundles.
Along the way, we study the cofibrant diffeological spaces, and conjecture that
every smooth manifold is cofibrant.

\subsection{Diffeological realization and the smooth singular simplicial set}\label{ss:adjoint-pair}

In this subsection, we use an adjoint pair between simplicial sets and diffeological
spaces to define the concepts of cofibration, fibration and weak equivalence, and
prove some basic properties.
We then prove one of our main results, which says that the smooth
homotopy groups of a fibrant diffeological space coincide with the
simplicial homotopy groups of its smooth singular simplicial set.
We conclude with some properties of the diffeological realization
functor and the smooth singular functor.

Here is a general theorem from~\cite{Mac}:

\begin{thm}
Given a small category $\mathcal{C}$,
a cocomplete category $\mathcal{D}$,
and a functor $F:\mathcal{C} \ra \mathcal{D}$,
there is an adjoint pair $L:\Pre(\mathcal{C}) \rightleftharpoons \mathcal{D}:R$
with $R(d)(c)=\mathcal{D}(F(c),d)$
and $L(X)= \colim_{\mathcal{C}(\blank,c) \ra X} F(c)$,
where $c$ is an object in $\mathcal{C}$, $d$ is an object in $\mathcal{D}$
and $X$ is a presheaf on $\mathcal{C}$.
\end{thm}

If we take $\mathcal{C}$ to be the simplicial category $\Delta$,
then the above theorem says that,
given a cosimplicial object in a cocomplete category $\mathcal{D}$
(that is, a functor $\Delta \ra \mathcal{D}$),
we get an adjoint pair $\sSet \rightleftharpoons \mathcal{D}$.

\begin{ex}\label{ex:adjoint-top}
If we take $F$ to be the functor $\Delta \ra \Top$
sending $\underline{n}$ to
$|\Delta^n|=\{(x_0,x_1,\ldots,x_n) \in \R^{n+1} \mid \sum x_i=1
\text{ and } x_i \geq 0
\text{ for each } i\}$ with the sub-topology from $\R^{n+1}$,
then we get the usual adjoint pair $|\blank|:\sSet \rightleftharpoons \Top:s$.
\end{ex}

\begin{de}\label{adjoint}
We write $\mathbb{A}^n=\{(x_0,x_1,\ldots,x_n) \in \R^{n+1} \mid \sum x_i=1\}$
with the sub-diffeology from $\R^{n+1}$.
It is diffeomorphic to $\R^n$, by forgetting the first coordinate, for example.
Just like the standard cosimplicial object in $\Top$,
$\mathbb{A}^\bullet$ is a cosimplicial object in $\Diff$.
Hence, we get an adjoint pair $|\blank|_D:\sSet \rightleftharpoons \Diff:S^D$.
We call $|\blank|_D$ the \dfn{diffeological realization functor}
and $S^D$ the \dfn{smooth singular functor}.
\end{de}

More precisely, by the above theorem, we know that
$S^D_n(X) = C^\infty(\mathbb{A}^n,X) \cong C^\infty(\R^n,X)$
and $|A|_D = \colim_{\Delta^n \ra A} \A^n$.
As is usual for geometric realizations, the latter can be
described more concretely.
Let $\sim$ be the equivalence relation on $\coprod_{n \in \N} \A^n \times A_n$
generated by $\A^n \times A_n \ni (a,x) \sim (a',x') \in \A^m \times A_m$
if there is a morphism $f:\underline{n} \ra \underline{m}$ in $\Delta$
such that $f_*(a)=a'$ and $f^*(x')=x$,
where $f_*:\A^n \ra \A^m$ and $f^*:A_m \ra A_n$ are induced from $f$.
Then $|A|_D=(\coprod_{n \in \N} \R^n \times A_n) / {\sim}$, with the quotient diffeology.

\begin{rem}
Instead of using these ``non-compact'' simplices $\A^n$, one could
use the obvious compact versions, with $0 \leq x_i \leq 1$ for each $i$.
In~\cite[Section 5]{He}, Hector defined singular and
geometric realization functors using the compact simplices.
We use the non-compact versions because the smooth maps from $\A^n$ to
a diffeological space $X$ are simply plots.
Moreover, with compact simplices, we don't know whether many of our
results would continue to hold.
\end{rem}

We can describe some important diffeological realizations explicitly.
The horn $\Lambda_k^n$ is the sub-simplicial set of $\Delta^n$
which omits the $n$-simplex and its $k^{th}$ face.
This is the coequalizer of its other $(n-1)$-dimensional faces
along their $(n-2)$-dimensional intersections.
Since diffeological realization is a left adjoint,
$|\Lambda^n_k|_D$ is the coequalizer of $n$ copies of $\A^{n-1}$
along $\binom{n}{2}$ copies of $\A^{n-2}$.
It is easy to see that all of the $|\Lambda^n_k|_D$'s are diffeomorphic to
$\Lambda^n:=\{(x_1,\ldots,x_n) \in \R^n \mid x_i = 0 \text{ for some } i \}$,
viewed as the coequalizer of the coordinate hyperplanes along their intersections,
with the coequalizer diffeology.

Similarly, the boundary $\partial \Delta^n$ can be described as a coequalizer,
and is diffeomorphic to
$\partial' \R^n := \{(x_1,\ldots,x_n) \in \R^n \mid x_i = 0
\text{ for some } i \text{ or } \sum x_i = 1 \}
= \Lambda^n \cup \{(x_1,\ldots,x_n) \in \R^n \mid \sum x_i = 1 \}$,
where both are equipped with the coequalizer diffeology.

\begin{rem}\label{coeqvssub}
Note that $\Lambda^n$ and $\partial' \R^n$ are not diffeological subspaces of $\R^n$ when $n \geq 2$.
Write $\Lambda^n_{sub}$ and $\partial' \R^n_{sub}$ for
the diffeological subspaces of $\R^n$ with the same underlying sets
as $\Lambda^n$ and $\partial' \R^n$, respectively.
Then we have smooth maps $\Lambda^n \ra \Lambda^n_{sub}$ and $\partial' \R^n \ra \partial' \R^n_{sub}$
which are both identity maps on the underlying sets.
\end{rem}

\begin{de}\label{LLP-RLP}
If $i : A \ra B$ and $f : X \ra Y$ are morphisms in a category such that
for every solid commutative square
\[
  \xymatrix{
    A \ar[r] \ar[d]_i     & X \ar[d]^f \\
    B \ar[r] \ar@{-->}[ur] & Y
  }
\]
a dotted morphism exists making the triangles commute, then we say that
$i$ has the \dfn{left lifting property} with respect to $f$ and that
$f$ has the \dfn{right lifting property} with respect to $i$.
\end{de}

\begin{de}\label{sSet-wfc}
A map $f : X \ra Y$ of simplicial sets is a
\dfn{weak equivalence} if its geometric realization is a homotopy equivalence in $\Top$,
a \dfn{cofibration} if each $X_n \to Y_n$ is a monomorphism, and
a \dfn{(Kan) fibration} if it has the right lifting property with
respect to the inclusions $\Lambda^n_k \hookrightarrow \Delta^n$.
\end{de}

It is well-known~\cite{GJ,Q} that these definitions give a
cofibrantly generated proper model structure on $\sSet$.

\begin{de}\label{wfc}
We call a morphism $X \ra Y$ in $\Diff$ a \dfn{weak equivalence}
(resp.\ \dfn{fibration})
if $S^D(X) \ra S^D(Y)$ is a weak equivalence (resp.\ fibration) in $\sSet$.
We call a morphism $X \ra Y$ in $\Diff$ a \dfn{cofibration}
if it has the left lifting property with respect to all maps which
are both weak equivalences and fibrations.
\end{de}

\begin{rem}
The above definition is motivated by the standard model structure
on $\Top$, for which a map $f$ is a weak equivalence (resp.\ fibration)
if and only if $s(f)$ is a weak equivalence (resp.\ fibration) in $\sSet$,
and for which the cofibrations are determined by a left lifting property.

More generally, given an adjoint pair $F : \cM \leftrightharpoons \cN : U$,
where $\cM$ is a cofibrantly generated model category and $\cN$ is a
complete and cocomplete category, one can often ``pull back'' the model
structure on $\cM$ along the right adjoint $U$ in the analogous
way using Kan's lifting lemma~\cite[Theorem~11.3.2]{Hi}.
There are conditions that must be checked, and in our situation we
are unable to verify condition (2) of the cited theorem.
In particular, we do not know how to show that the pushout of a
map which is both a cofibration and a weak equivalence is a weak equivalence.
Nevertheless, we will show that our definitions can be used to study the
homotopy theory of diffeological spaces and in particular that they capture important
properties of the smooth singular simplicial set of a diffeological space.

In~\cite[Section 5]{He}, Hector defined the notion of ``Kan fibration'',
which is the analog of our notion of fibration, but using compact simplices.
We are not sure whether either notion implies the other.
\end{rem}

In any of the above categories, a map is a \dfn{trivial (co)fibration} if
it is both a weak equivalence and a (co)fibration.
An object is \dfn{cofibrant} if the unique map from the empty object is a cofibration,
and is \dfn{fibrant} if the unique map to a point is a fibration.
Thus a diffeological space is fibrant if and only if $S^D(X)$ is fibrant,
which is also known as being a Kan complex.
We will see that for fibrant diffeological spaces, $S^D(X)$ captures
the smooth homotopical information of $X$.
In order to prove this,
we use the following lemma connecting smooth homotopy and simplicial homotopy:

\begin{lem}\label{shtvssimhtp}
The functor $S^D:\Diff \ra \sSet$ sends smoothly homotopic maps to simplicially homotopic maps.
\end{lem}
\begin{proof}
Suppose that $f,g:X \ra Y$ are smoothly homotopic, so that
we have the following commutative diagram in $\Diff$:
\[
\xymatrix{X \times \{0\} \ar[d] \ar[dr]^f \\ X \times \R \ar[r] & Y \\ X \times \{1\}. \ar[u] \ar[ur]_g}
\]
Since $S^D$ is a right adjoint,
we have the following commutative diagram in $\sSet$:
\[
\xymatrix{& S^D(X) \times \Delta^0 \ar[dl] \ar[d] \ar[dr]^{S^D(f)} \\
S^D(X) \times \Delta^1 \ar[r]^-{1 \times \nu}
 & S^D(X) \times S^D (\R) \ar[r] & S^D(Y) \\
 & S^D(X) \times \Delta^0, \ar[ul] \ar[u] \ar[ur]_{S^D(g)}}
\]
where $\nu$ corresponds to the projection $\A^1 \ra \R$ onto the second coordinate.
Thus $S^D(f)$ and $S^D(g)$ are simplicially homotopic.
\end{proof}

We call a diffeological space $X$ \dfn{smoothly contractible},
if the identity map $X \ra X$ is smoothly homotopic to a constant map $X \ra X$.
Therefore, if a diffeological space $X$ is smoothly contractible,
then $X \ra \R^0$ is a weak equivalence.
For example, since both $\Lambda^n$ and $\Lambda^n_{sub}$ are linearly
contractible to the origin, the map $\Lambda^n \ra \Lambda^n_{sub}$
introduced in Remark~\ref{coeqvssub} is a weak equivalence.

Here is an important property of fibrant diffeological spaces:

\begin{thm}\label{fib-hg-compare}
Let $(X,x)$ be a pointed diffeological space with $X$ fibrant.
Then there is a natural isomorphism $\pi_n^D(X,x) \cong \pi_n^s(S^D(X),x)$.
\end{thm}
\begin{proof}
Since $X$ is a fibrant diffeological space,
$S^D(X)$ is a Kan complex.

For $n=0$, the result is straightforward, since
$\pi_0^s(S^D(X))$ can be described as the coequalizer of
\[
\xymatrix{(S^D(X))_1 \ar@< 2pt>[r]^{d_0} \ar@<-2pt>[r]_{d_1} & (S^D(X))_0.}
\]
Therefore, $\pi_0^D(X) \cong \pi_0^s(S^D(X))$.

For $n \geq 1$,
the $n^{th}$ simplicial homotopy group $\pi_n^s(S^D(X),x)$ of $(S^D(X),x)$ is defined
to be the set of simplicial homotopy classes of maps of pairs
$(\Delta^n, \partial \Delta^n) \to (S^D(X), x)$.

In~\eqref{item:5} of Theorem~\ref{equidef}, we proved that $\pi_n^D(X,x)$ bijects
with $[(\mathbb{A}^n,\partial \mathbb{A}^n),(X,x)]$.
We define
\[
\alpha: [(\mathbb{A}^n,\partial \mathbb{A}^n),(X,x)] \ra \pi_n^s(S^D(X),x)
\]
by $\alpha([f])=[f]$, where an $n$-simplex of $S^D(X)$ is identified with
the corresponding map $\Delta^n \to S^D(X)$.
The map $\alpha$ is well-defined by Lemma~\ref{shtvssimhtp},
and it is clear that $\alpha$ is surjective.

We now show that $\alpha$ is injective.
Let $[f],[g] \in [(\mathbb{A}^n,\partial \mathbb{A}^n),(X,x)]$ be such that $\alpha([f])=\alpha([g])$.
Since $S^D(X)$ is a Kan complex,
there exists $F \in C^\infty(\mathbb{A}^{n+1},X)$ such that
$F(x_0,\ldots,x_{n-1},0,x_{n+1})=f(x_0,\ldots,x_{n-1},x_{n+1})$,
$F(x_0,\ldots,x_n,0)=g(x_0,\ldots,x_n)$ and $F(x_0,\ldots,x_{n+1})=x$ if some other $x_i=0$.
Then the composite $F \circ \beta$, with
$\beta:\mathbb{A}^n \times \R \ra \mathbb{A}^{n+1}$ defined by
\[
\beta(x_0,\ldots,x_n,t)=(x_0,\ldots,x_{n-1},tx_n,(1-t)x_n),
\]
implies that $[f]=[g]$ in $[(\mathbb{A}^n,\partial \mathbb{A}^n),(X,x)]$.

Finally, we will show that $\alpha$ is a group homomorphism for $n \geq 1$.
In Theorem~\ref{equidef}, we showed that the restriction map
$i^*: [(\mathbb{A}^n,\partial_{\eps} \A^n),(X,x)] \ra [(\mathbb{A}^n,\partial \mathbb{A}^n),(X,x)]$
is an isomorphism.
Thus we can assume that we are given
$[f],[g] \in [(\mathbb{A}^n,\partial_{\eps} \A^n),(X,x)]$.

First we compute the product of $\alpha([f])$ and $\alpha([g])$ in $\pi_n^s(S^D(X),x)$.
By projecting $\A^{n+1}$ down to the union of its $(n-1)^{th}$ and $(n+1)^{th}$ faces
and composing with $f$ and $g$ on those faces, one obtains a map $h : \A^{n+1} \to X$
which can be used to compute the product.
This illustrates the projection in the case $n = 1$\mynobreakpar
\begin{center}
\vspace*{5pt}
\includegraphics[height=1in,width=1.15in]{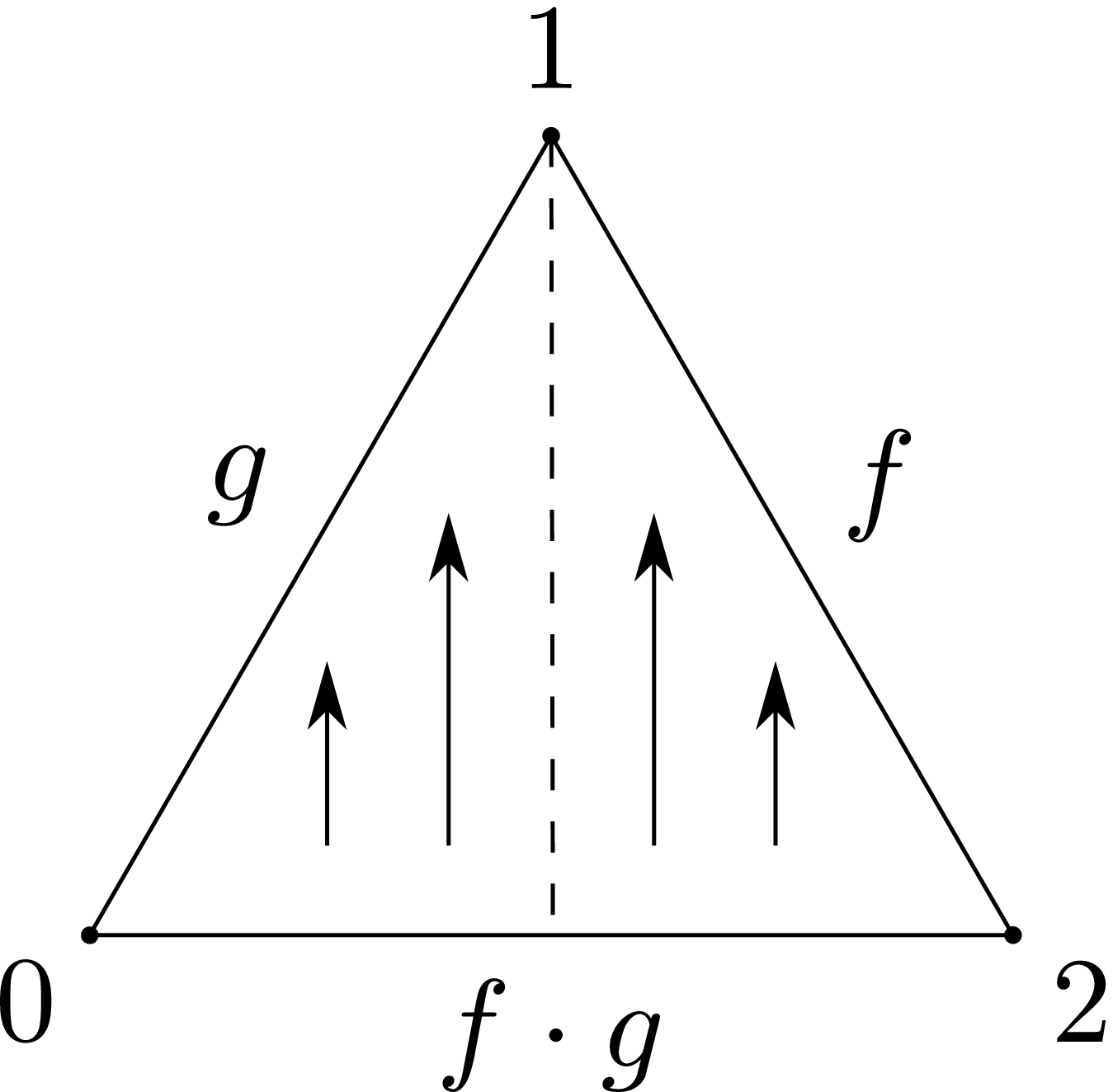}
\vspace*{-5pt}
\end{center}
and the general formula for the composite $h$ is
\[
h(x_0,\ldots,x_{n+1})=\begin{cases}
f(x_0,\ldots,x_{n-2},x_n+2x_{n-1},x_{n+1}-x_{n-1}),
& \text{if } x_{n+1} \geq x_{n-1} \\
g(x_0,\ldots,x_{n-2},x_{n-1}-x_{n+1},x_n+2x_{n+1}),
& \text{if } x_{n+1} \leq x_{n-1} .
\end{cases}
\]
The projection map is not smooth, but $h$ is smooth because $f$ and $g$
are constant near their boundaries.
It is straightforward to check that $d_i h = x$ for $0 \leq i < n-1$,
$d_{n-1} h = f$ and $d_{n+1} h = g$, and so by the definition of the product
in $\pi_n^s(S^D(X),x)$, $d_n h$ represents the product of $\alpha([f])$ and $\alpha([g])$.

Note that $d_n h$ is a certain juxtaposition of scaled and translated versions of $f$ and $g$.
On the other hand, the product of $[f]$ and $[g]$ in $[(\mathbb{A}^n,\partial_{\eps} \A^n),(X,x)]$
is given by first regarding $f$ and $g$ as maps $(\mathbb{R}^n,\partial_{\eps} \R^n) \to (X,x)$,
juxtaposing and scaling them as usual, and then scaling further to obtain a map
$(\mathbb{A}^n,\partial_{\eps} \A^n) \to (X,x)$,
as described in the proof of $(\ref{item:6}) \iff (\ref{item:2})$ in Theorem~\ref{equidef}.
One can see that the result is homotopic to $d_n h$ using techniques similar
to those used in Theorem~\ref{equidef}.
\end{proof}

\begin{cor}
A map $f : X \to Y$ between fibrant diffeological spaces is a weak equivalence
if and only if it induces an isomorphism on all smooth homotopy groups for all
basepoints.
\end{cor}

The above results highlight the importance of understanding which diffeological
spaces are fibrant.
This is discussed in Subsection~\ref{ss:fibrant}.

We conclude this section with some observations about the diffeological realization
functor and the smooth singular functor.
Unlike the usual geometric realization functor $|\blank|:\sSet \ra \Top$, we have:

\begin{prop}\label{noncomm}
The functor $|\blank|_D: \sSet \ra \Diff$ does not commute with finite products.
\end{prop}
\begin{proof}
For simplicial sets $X$ and $Y$,
we have a natural map $|X \times Y|_D \ra |X|_D \times |Y|_D$
induced from the projections.
One can show that this map is surjective.
However, it is not always a diffeomorphism.
For example, it is easy to see that $|\Delta^1 \times \Delta^1|_D$ is the pushout of
\[
\xymatrix{|\Delta^1|_D \ar[r]^{|d^0_*|} \ar[d]_{|d^2_*|} & |\Delta^2|_D,
\\ |\Delta^2|_D}
\]
and hence is not diffeomorphic to $\R^2 \cong |\Delta^1|_D \times |\Delta^1|_D$.
(In fact, $|\Delta^1 \times \Delta^1|_D \cong \Lambda^2 \times \R$,
and the natural map is not even injective in this case.)

As another example, let $A$ be the simplicial set whose non-degenerate simplices are
\[ \xymatrix{\bullet \ar@/^/[r] \ar@/_/[r] & \bullet} \]
Then $\R \times |A|_D$ is not the diffeological realization of any simplicial set $B$.
If it were, $B$ would have no non-degenerate simplices of dimension greater than two,
and the ways in which the $2$-simplices were attached would be visible in the ``seams''
that arise when copies of $\R^2$ are glued along lines or when lines are collapsed
to points.
The seams in $\R \times |A|_D$ consist of two parallel lines, but two edges of
a triangle always intersect, so this cannot arise as $|B|_D$.
\end{proof}

In the next two propositions,
we compare the three adjoint pairs $|\blank|_D:\sSet \rightleftharpoons \Diff:S^D$,
$D:\Diff \rightleftharpoons \Top:C$ and $|\blank|:\sSet \rightleftharpoons \Top:s$.
These results require some techniques from model categories,
but are not needed in the rest of the paper.

\begin{prop}\label{3adjoint1}
Given any topological space $A$,
there is a weak equivalence between $S^D(C(A))$ and $sA$ in $\sSet$.
\end{prop}
\begin{proof}
For any topological space $A$,
$S^D(C(A))=C^\infty(\mathbb{A}^\bullet,C(A))=\Top(D(\mathbb{A}^\bullet),A)$,
and $sA=\Top(|\Delta^\bullet|,A)$.
To compare these, we will make use of
the Reedy model structure on $\Top^\Delta$,
drawing upon many results from~\cite[Chapters~15 and~18]{Hi}.
Note that every topological space is fibrant in the standard model structure of $\Top$,
both $D(\mathbb{A}^\bullet)$ and $|\Delta^\bullet|$ are cosimplicial resolutions of a point in $\Top$,
and the natural inclusion map $i:|\Delta^\bullet| \ra D(\mathbb{A}^\bullet)$
is a Reedy weak equivalence in $\Top^\Delta$
(since for any $n \in \N$, both $|\Delta^n|$ and $D(\mathbb{A}^n)$ are contractible).
Therefore, $i^*:S^D(C(A)) \ra sA$ is a weak equivalence of fibrant simplicial sets,
by~\cite[Corollary 16.5.5(1)]{Hi}.
\end{proof}

\begin{prop}\label{3adjoint2}
Given any simplicial set $X$,
there is a weak equivalence between $D(|X|_D)$ and $|X|$ in $\Top$.
\end{prop}
\begin{proof}
Recall that $|X| = \colim_{\Delta^n \to X} |\Delta^n|$ (Example~\ref{ex:adjoint-top}).
Since $D$ is a left adjoint, it commutes with colimits, and so we have that
$D(|X|_D) = \colim_{\Delta^n \to X} D(\A^n)$.
As described in the proof of the previous proposition,
both $|\Delta^\bullet|$ and $D(\mathbb{A}^\bullet)$ are cosimplicial resolutions of a point in $\Top$,
and so are Reedy cofibrant,
and the natural inclusion map $|\Delta^\bullet| \ra D(\mathbb{A}^\bullet)$
is a Reedy weak equivalence in $\Top^\Delta$.
By~\cite[Proposition~16.5.6(1) and Corollary 7.7.2]{Hi},
it follows that the induced map $|X| \to D(|X|_D)$ is a weak equivalence in $\Top$.
(One can also use~\cite[Lemma~A.5]{Se}, which is a less abstract form of the same result.)
\end{proof}

\subsection{Cofibrant diffeological spaces}\label{ss:cofibrant}

In this subsection, we study the cofibrant diffeological spaces.
After some preliminary observations, we prove one of the factorization
axioms of a model category, which implies that every diffeological
space has a functorial cofibrant replacement.
Then we give examples of cofibrant diffeological spaces, culminating
in the proof that fine diffeological vector spaces and $S^1$ are cofibrant.

We begin with some basic observations:

By the adjunction $|\blank|_D:\sSet \rightleftharpoons \Diff :S^D$
and Definition~\ref{wfc},
$X \ra Y$ is a fibration in $\Diff$ if and only if
it has the right lifting property with respect to $\Lambda^n \ra \R^n$ for all $n \in \Z^+$,
and $X \ra Y$ is a trivial fibration in $\Diff$ if and only if
it has the right lifting property with respect to $\partial' \R^n \ra \R^n$ for all $n \in \N$.
In particular, taking $n=0$, we see that all trivial fibrations are surjective.

Also, if a smooth map $f:A \ra B$ is the diffeological realization of a trivial cofibration in $\sSet$,
and $g:X \ra Y$ is a fibration in $\Diff$, then any commutative solid diagram
\[
\xymatrix{A \ar[r] \ar[d]_f & X \ar[d]^g \\ B \ar[r] \ar@{.>}[ur] & Y}
\]
in $\Diff$ has a smooth lift.

\begin{prop}\label{pr:cofs}
The functor $|\blank|_D:\sSet \ra \Diff$ preserves cofibrations.
The class of cofibrations in $\Diff$ is closed under
isomorphisms, pushouts, smooth retracts and (transfinite) compositions.
In particular, the diffeological realization of any simplicial set is cofibrant.
\end{prop}
\begin{proof}
This is formal.
\end{proof}

\begin{prop}\label{factorization}
Every smooth map $f$ in $\Diff$ has a functorial factorization as
$f=\alpha(f) \circ \beta(f)$ with $\alpha(f)$ a trivial fibration and $\beta(f)$ a cofibration.
\end{prop}
\begin{proof}
We claim that every diffeological space is small, in the sense
used in the small object argument;
see, for example,~\cite[Definition~2.1.3]{Ho}.
One can prove smallness by a straighforward argument, directly
from the definitions~\cite[Theorem~2.1.3]{W1}.
Or one can use that $\Diff$ is equivalent to the category of
concrete sheaves over a concrete site (see Remark~\ref{rem:concrete})
and then apply~\cite[Theorem~C2.2.13]{J}, which says
the category of concrete sheaves over a concrete site is locally presentable.

In any case, the result then follows by applying the small object
argument (e.g.,~\cite[Theorem~2.1.14]{Ho})
to the set $I=\{\partial' \R^n \ra \R^n \mid n \in \N\}$.
\end{proof}

By applying Proposition~\ref{factorization} to the map $\emptyset \to X$,
we obtain the following immediate consequence:

\begin{cor}\label{cofrep}
Every diffeological space has a functorial cofibrant replacement.
\end{cor}

\begin{ex}\
\begin{enumerate}
\item $\Lambda^n \ra \R^n$ for any $n \in \Z^+$
and $\partial' \R^m \ra \R^m$ for any $m \in \N$ are all cofibrations,
since they are diffeological realization of cofibrations in $\sSet$.

\item $\R^n$ is cofibrant for any $n \in \N$, since $\R^n=|\Delta^n|_D$.

\item $\Lambda^2$ is cofibrant, since it is the pushout of
\[
\xymatrix{\R^0 \ar[r] \ar[d] & \R \\ \R.}
\]

\item More generally, all $\Lambda^n=|\Lambda^n_k|_D$
and $\partial' \R^n=|\partial \Delta^n|_D$ are cofibrant.
This can also be seen by building them as pushouts along the cofibrations in the above examples,
and along the way, we obtain other interesting cofibrations and cofibrant objects.
For example, $\vee_{i=1}^n \R$ is cofibrant for any $n$.
\end{enumerate}
\end{ex}

\begin{ex}
A \dfn{diffeological vector space} is an $\R$-vector space with a diffeology
such that the addition and the scalar multiplication maps are smooth.
Any $\R$-vector space $V$ has a smallest diffeology making it a diffeological vector space,
and this is called the \dfn{fine diffeology}; see~\cite[Chapter~3]{I2}.
The fine diffeology is generated by all linear maps from finite dimensional $\R$-vector spaces to $V$.
A diffeological vector space with the fine diffeology is called a \dfn{fine diffeological vector space}.

For example, the colimit in $\Diff$ of the natural inclusions
\[
\xymatrix{\R^0 \ar[r] & \R^1 \ar[r] & \cdots \ar[r] & \R^n \ar[r] & \R^{n+1} \ar[r] & \cdots }
\]
is a fine diffeological vector space.
Since each of the inclusions is a cofibration, the colimit is cofibrant.
\end{ex}

More generally, we have:

\begin{prop}\label{pr:fine}
Every fine diffeological vector space is cofibrant.
\end{prop}

\begin{proof}
Let $V$ be an arbitrary fine diffeological vector space.
Choose a basis $\{v_i\}_{i \in I}$ for $V$, and consider the
category $\cI$ of finite subsets of $I$ and inclusions.
There is a functor $F: \cI \to \Diff$ sending a finite subset $J$
to the span of $\{v_j\}_{j \in J}$, with the sub-diffeology
(which is the standard diffeology).
The colimit of $F$ is $V$, essentially by the definition of the fine diffeology.

For each finite subset $J$, the \dfn{latching map} is the map
$\colim_{J' \subset J} F(J') \to F(J)$, where the colimit is over \emph{proper}
subsets of $J$.
This map is diffeomorphic to the map $\Lambda^n \ra \R^n$, where $n = |J|$,
and so it is a cofibration.
Thus, by a standard induction (see, for example, the proof of Proposition~5.1.4 in~\cite{Ho}),
we conclude that $\colim F \cong V$ is cofibrant.
\end{proof}

\begin{ex}\label{hatS1}
The pushout of
\[
\xymatrix{\partial' \R \ar[r] \ar[d] & \R^0 \\ \R}
\]
will be denoted by $\hat{S}^1$ and is cofibrant.

Clearly $\hat{S}^1$ is not diffeomorphic to $S^1$,
because $\hat{S}^1$ has ``tails''.
But even the diffeological subspace of $\hat{S}^1$ with the tails removed
is not diffeomorphic to $S^1$,
because of the point where the gluing occurs.

\begin{figure}[ht]
\begin{minipage}[b]{0.45\linewidth}
\centerline{
\xy
(0,0)*{}="A";
(10,0)*{}="B";
(5,3.8)*{\bullet};
"A";"B" **\crv{(15,10) & (5,15) & (-5,10)};
\endxy
}
\caption{$\hat{S}^1$}
\end{minipage}
\hspace{0cm}
\begin{minipage}[b]{0.45\linewidth}
\centerline{
\xy
(0,0)*{\bullet}="A";
"A";"A" **\crv{(10,10) & (0,15) & (-10,10)};
\endxy
}
\caption{$\hat{S}^1$ with tails removed}
\end{minipage}
\end{figure}
\end{ex}

As we have seen in Proposition~\ref{pr:cofs}, the diffeological realization of
any simplicial set is cofibrant.
However, the spaces built in this way have tails and gluing points, and so
cannot be smooth manifolds.
Nevertheless, we are able to show that $S^1$ is a smooth retract of such a
realization, and therefore that it is cofibrant.

\begin{prop}\label{S1cof}
$S^1$ is cofibrant.
\end{prop}

\begin{proof}
Let $X$ be the simplicial set whose non-degenerate simplices are:
\[
\centerline{
\xy
(0,0)*+{x}="E";(20,0)*+{x}="F";(10,17.32)*+{y}="G";(30,17.32)*+{y}="H";(10,5.77)*+{A};(20.7,12)*+{B};
{\ar   "E";"F"};
{\ar^a "E";"G"};
{\ar^b "G";"F"};
{\ar_a "F";"H"};
{\ar   "G";"H"};
\endxy
}
\]
Note that the left edge is identified with the right edge, forming a cylinder.
So $|X|_D$ consists of two copies of $\A^2$ glued along two lines $a$ and $b$, and is cofibrant.
There is a map $|X|_D \to S^1$ sending a point to $e^{i\pi \theta}$,
where $\theta$ is the point's horizontal position on the page.
More precisely, let $\sigma_A : \A^{2} \to S^1$ be defined by $\sigma_A(x,y,z) = e^{i\pi(z-x)}$,
and let $\sigma_B : \A^{2} \to S^1$ be defined by $\sigma_B(x,y,z) = e^{i\pi(1+z-x)}$.
The affine functions $z-x$ and $1+z-x$ take the values shown here:
\[
\centerline{
\xy
(0,0)*+{-1}="E";(20,0)*+{1}="F";(10,17.32)*+{0}="G";(30,17.32)*+{2}="H";(10,5.77)*+{A};(20.7,12)*+{B};
{\ar   "E";"F"};
{\ar^a "E";"G"};
{\ar^b "G";"F"};
{\ar_a "F";"H"};
{\ar   "G";"H"};
\endxy
}
\]
Since those values are used modulo 2, $\sigma_A$ and $\sigma_B$ agree
on the identified lines $a$ and $b$, and thus define a smooth function $\sigma : |X|_D \to S^1$.

It will suffice to prove that $\sigma$ has a smooth section $s : S^1 \to |X|_D$.
In order to map into $|X|_D$, we must be careful to tangentially approach
the lines along which the gluing occurs.
Here is what our embedding will look like:\mynobreakpar
\begin{center}
\includegraphics[width=3in]{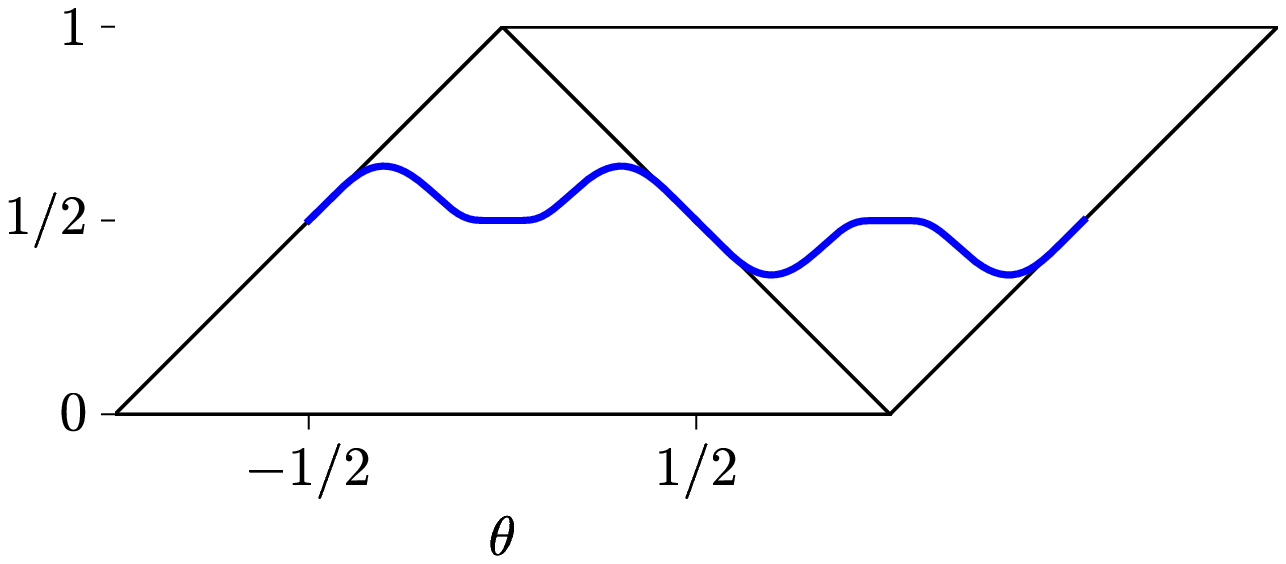}
\vspace*{-5pt}
\end{center}
The function describing the height of the portion on the left is given by
\[
  R(\theta) = \phi(2 |\theta|) \, (1 - |\theta|) + (1 - \phi(2 |\theta|)) \, \frac{1}{2},
\]
for $-1/2 \leq \theta \leq 1/2$, where $\phi$ is a cut-off function.
This blends between the function $1 - |\theta|$, which gives the edges
$a$ and $b$ of the left simplex for $|\theta|$ near $1/2$,
and the constant function $1/2$, near $\theta = 0$.
For $-1/2 \leq \theta \leq 1/2$, the unique point $(x,y,z) \in \A^2$
such that $\sigma_A(x,y,z) = e^{i\pi (z-x)} = e^{i\pi \theta}$ and $y = R(\theta)$ is given by
$c(\theta) := ((1-\theta-R(\theta))/2,\, R(\theta),\, (1+\theta-R(\theta))/2)$.
A similar argument works for the second simplex, and so our section
$s : S^1 \to |X|_D$ is given by
\[
  s(e^{i\pi \theta}) = \begin{cases}
                        (c(\theta),\, A ),    &          - \frac{1}{2} \leq \theta \leq \frac{1}{2} , \\
                        (c(\theta-1),\, B ),  & \phantom{-}\frac{1}{2} \leq \theta \leq \frac{3}{2} .
                      \end{cases}
\]
The section $s$ is smooth because it approaches the edges tangentially.
\end{proof}

Our argument actually shows that the natural map
$|S^D(S^1)|_D \to S^1$
has a smooth section.
We conjecture that this is true for any smooth manifold,
and therefore that every smooth manifold is cofibrant.

In Example~\ref{noncof} we will see that there exist non-cofibrant diffeological spaces.

\subsection{Fibrant diffeological spaces}\label{ss:fibrant}

In this subsection, motivated by Theorem~\ref{fib-hg-compare},
we study the fibrant diffeological spaces.
After some preliminaries,
we show that diffeological bundles with fibrant fibers are fibrations,
which allows us to show that the irrational tori are not cofibrant.
We then prove the elementary fact that every diffeological group
is fibrant, and use this to
show that any homogeneous diffeological space is fibrant.
The fact that diffeomorphism groups are diffeological groups then
implies that every smooth manifold is fibrant, one of our key results.
We are also able to show that many function spaces are fibrant.
We then give examples of non-fibrant diffeological spaces, in
particular showing that a smooth manifold with boundary is not fibrant.

\begin{prop}[Right Proper]
Let
\[
\xymatrix{W \ar[r]^h \ar[d] & X \ar[d]^f \\ Z \ar[r]_g & Y}
\]
be a pullback diagram in $\Diff$ with $f$ a fibration and $g$ a weak equivalence.
Then $h$ is also a weak equivalence.
\end{prop}
\begin{proof}
This follows from the right properness of the standard model structure on $\sSet$.
\end{proof}

\begin{lem}\label{copfib}
Fibrant diffeological spaces are closed under coproducts in $\Diff$,
and if $X$ is fibrant, then so is each path component.
\end{lem}
\begin{proof}
This is because both $D(\Lambda^n)$ and $D(\R^n)$ are connected.
\end{proof}

\begin{prop}
The class of fibrations in $\Diff$ is closed under
isomorphisms, pullbacks, smooth retracts and finite compositions.
\end{prop}
\begin{proof}
This is formal.
\end{proof}

As an immediate consequence of this proposition, we have

\begin{cor}
Let $f:X \ra Y$ be a fibration in $\Diff$.
Then every fiber of $f$ is fibrant, that is,
for any $y \in Y$, $f^{-1}(y)$ with the sub-diffeology from $X$ is fibrant.
\end{cor}

\begin{prop}\label{diffbundfibration}
Any diffeological bundle with fibrant fiber is a fibration.
\end{prop}
\begin{proof}
Let $f:X \ra Y$ be a diffeological bundle with fibrant fiber $F$.
Given any commutative diagram in $\Diff$
\[
\xymatrix{\Lambda^n \ar[d]_a \ar[r]^b & X \ar[d]^f \\ \R^n \ar[r]_c & Y,}
\]
we have the following pullback diagram in $\Diff$
\[
\xymatrix{\R^n \times F \ar[d]_{\pi_1} \ar[r]^-{d} & X \ar[d]^f \\ \R^n \ar[r]_c & Y,}
\]
where we have used Proposition~\ref{pr:global-plot} to see that the pullback is trivial.
Therefore, we have the following commutative diagram:
\[
\xymatrix{\Lambda^n \ar@/_/[ddr]_a \ar@/^/[drr]^b \ar@{.>}[dr]|-{(a,e)} \\
& \R^n \times F \ar[d]_{\pi_1} \ar[r]^{~~~d} & X \ar[d]^f \\ & \R^n \ar[r]_c & Y.}
\]
Let $g:\R^n \ra F$ be any smooth map and consider the smooth section $(1,g):\R^n \ra \R^n \times F$.
Then $f \circ d \circ (1,g) \circ \pi_1 = c \circ \pi_1 \circ (1,g) \circ \pi_1 = c \circ \pi_1$,
and by the surjectivity of $\pi_1$, we have the following commutative triangle
\[
\xymatrix{& X \ar[d]^f \\ \R^n \ar[ur]^{d \circ (1,g)} \ar[r]_c & Y.}
\]
We also want the triangle
\[
\xymatrix{\Lambda^n \ar[r]^b \ar[d]_a & X \\ \R^n, \ar[ur]_{d \circ (1,g)}}
\]
to commute, which requires us to pick the smooth map $g$ carefully.
Since $F$ is fibrant, we choose $g$ to be a lifting of
\[
\xymatrix{\Lambda^n \ar[d]_a \ar[r]^e & F \\ \R^n. \ar@{.>}[ur]_g}
\]
Then, for any $x \in \Lambda^n$,
we have $d \circ (1,g) \circ a(x) = d(a(x),g \circ a(x)) = d(a(x),e(x))=(d \circ (a,e))(x)=b(x)$.
\end{proof}

\begin{ex}\label{noncof}
Not every diffeological space is cofibrant.
For example, the irrational torus $T^2_\theta$ is not cofibrant.
To see this, first recall that the quotient map $T^2 \ra T^2_\theta$ is a trivial fibration.
But we saw in Example~\ref{irrtorus} that the identity map $T^2_\theta \ra T^2_\theta$
has no smooth lift to a map $T^2_\theta \to T^2$.
See~\cite[Example~6.8(2)]{W2} for an alternative proof of this example.
\end{ex}

\begin{prop}\label{diffgrpfib}
Every diffeological group is fibrant.
\end{prop}
\begin{proof}
The right adjoint of an adjoint pair between two categories with finite products
always sends group objects to group objects.
The group objects in $\Diff$ and in $\sSet$ are precisely
diffeological groups and simplicial groups, respectively,
and Moore's lemma (see, for example,~\cite[Lemma I.3.4]{GJ}) says that every simplicial group is fibrant in $\sSet$.
Hence the result follows.
\end{proof}

\begin{ex}
Here is a more concrete way to see that every diffeological abelian group $A$ is fibrant.
Given a solid diagram
\[
\xymatrix{\Lambda^n \ar[r]^F \ar[d] & A \\ \R^n \ar@{.>}[ur]_{\tilde{F}}}
\]
in $\Diff$, define the extension $\tilde{F}$ directly as follows.
For any $0 \leq k<n$ and $1 \leq i_1<\cdots<i_k \leq n$,
write $P_{i_1,\ldots,i_k}:\R^n \ra \Lambda^n$ for the orthogonal projection
onto the subspace where $x_i = 0$ for all $i \not\in \{ i_1, \ldots, i_k \}$.
When $k=0$, this is the constant map $\R^n \ra \Lambda^n$ sending everything to $0$.
All of these projections are clearly smooth.
Then the smooth map
\[
\tilde{F}=\sum_{k=0}^{n-1} \sum_{1 \leq i_1<\cdots<i_k \leq n} (-1)^{n-k+1} F \circ P_{i_1,\ldots,i_k}
\]
is an extension of $F$.
\end{ex}

\begin{ex}\
\begin{enumerate}
\item Every Lie group is fibrant.

\item Every irrational torus is fibrant.

\item Let $G$ be a diffeological group.
Then $C^\infty(X,G)$ is also a diffeological group for any diffeological space $X$,
and is therefore fibrant.
Similarly, for any $x_0 \in X$, the pointed mapping space $C^{\infty}((X,x_0), (G,e))$
is a diffeological group when given the sub-diffeology, and is thus fibrant.
Since there is a diffeomorphism $(G,e) \cong (G, g_0)$ for any $g_0 \in G$,
the same is true with $e$ replaced by $g_0$.
\end{enumerate}
\end{ex}

\begin{de}
Let $G$ be a diffeological group and let $H$ be a subgroup of $G$.
Then the set $G/H$ of left (or right) cosets, with the quotient diffeology,
is called a \dfn{homogeneous} diffeological space.
\end{de}

\begin{thm}\label{homogeneousfibrant}
Every homogeneous diffeological space is fibrant.
\end{thm}
\begin{proof}
Given $b:\Lambda^n \ra G/H$,
let $a:\R^0 \ra G$ be defined by $a(0) \in \pi^{-1}(b(0,\ldots,0))$,
where $\pi:G \ra G/H$ is the quotient map.
Then we have the following smooth liftings:
\[
\xymatrix{\R^0 \ar[r]^a \ar[d] & G \ar[d]^\pi \\ \Lambda^n \ar[r]^(.4)b \ar@{.>}[ur]^{\alpha} \ar[d] & G/H
\\ \R^n. \ar@{.>}[uur]^(.32)*!/^-2.5pt/{\labelstyle \beta} \ar@{.>}[ur]_{\gamma}}
\]
The lifting $\alpha$ exists because
$\R^0 \ra \Lambda^n$ is the diffeological realization of a trivial cofibration in $\sSet$
and $\pi$ is a fibration in $\Diff$ by Proposition~\ref{diffbundfibration}.
The lifting $\beta$ exists because $G$ is fibrant.
And $\gamma=\pi \circ \beta$ is easily seen to be the required lifting.
\end{proof}

\begin{rem}
The proof of this theorem shows that
if a smooth map $X \ra Y$ is a fibration in $\Diff$ and a surjective set map,
with $X$ fibrant,
then $Y$ is also fibrant.
\end{rem}

\begin{cor}\label{mfdfibrant}
Every smooth manifold is fibrant.
\end{cor}
\begin{proof}
Use Theorem~\ref{homogeneousfibrant}, Lemma~\ref{copfib}, and the fact that
the homogeneous diffeological space $\Diff(M)/\stab(M,x)$ is diffeomorphic to $M$ (\cite{Do}),
where $M$ is an arbitrary connected smooth manifold, $x \in M$,
and $\stab(M,x)=\{f \in \Diff(M) \mid f(x)=x\}$ is a subgroup of $\Diff(M)$.
\end{proof}

In Corollary~\ref{cofrep} we showed that for every diffeological space $X$
there is a trivial fibration from a cofibrant diffeological space $\tilde{X}$ to $X$.
Thus, if $X$ is fibrant, $\tilde{X}$ is both cofibrant and fibrant.
In particular, if $M$ is a smooth manifold,
then $\tilde{M}$ is both cofibrant and fibrant, and is weakly equivalent to $M$.

\begin{rem}
In Corollary~\ref{cpt}, we prove a general result about fibrancy of function spaces,
which gives a second proof that smooth manifolds are fibrant, as well as a proof
that $C^{\infty}(S^1, M)$ is fibrant for any smooth manifold $M$.
\end{rem}

We conjecture that if $X$ is a cofibrant diffeological space and $Y$ is a fibrant diffeological space,
then $C^\infty(X,Y)$ is fibrant.
We prove the following special case.

\begin{prop}\label{fsfib}
Let $Y$ be a fibrant diffeological space.
Then $C^\infty(\R^m,Y)$ is fibrant for any $m \in \N$.
\end{prop}

More generally, the proof below shows that
if $f:X \ra Y$ is a fibration in $\Diff$, then so is
$f_*:C^\infty(\R^m,X) \ra C^\infty(\R^m,Y)$ for any $m \in \N$.

\begin{proof}
It is easy to see that
for any $n,m \in \N$, $i \times 1:\Lambda^n \times \R^m \ra \R^n \times \R^m$
is the diffeological realization of the (trivial) cofibration
$\cup_{i=1}^n d^i(\Delta^{n+m}) \hookrightarrow \Delta^{n+m}$ in $\sSet$.
Then the result follows by the cartesian closedness of $\Diff$.
\end{proof}

\begin{rem}
One might expect that the above proposition generalizes immediately to the case of $C^\infty(|A|_D,Y)$,
with $A$ an arbitrary simplicial set and $Y$ a fibrant diffeological space.
However, the above proof does not go through in general,
since the diffeological realization does not commute with finite products,
as shown in Proposition~\ref{noncomm}.
\end{rem}

\begin{thm}\label{Dopnfibrant}
Let $X$ be a fibrant diffeological space.
Then every $D$-open subset of $X$ with the sub-diffeology is also fibrant.
\end{thm}
\begin{proof}
We write $\Lambda^n_{sub}$ for $\Lambda^n$ with the sub-diffeology from $\R^n$.
Let $U$ be an open neighborhood of $\Lambda^n_{sub}$ in $\R^n$,
and let $Y$ be an arbitrary diffeological space.
We claim that if a smooth map $f: \Lambda^n_{sub} \ra Y$ has a smooth extension $g: U \ra Y$,
then $f$ has a smooth extension $h: \R^n \ra Y$.

Here is the proof of the claim.
Since $\Lambda^n_{sub}$ is the set of all coordinate hyperplanes in $\R^n$,
and $U$ is an open neighborhood of $\Lambda^n_{sub}$ in $\R^n$,
there exists an open neighborhood $V$ of $\Lambda^n_{sub}$ in $\R^n$ with $V \subseteq U$
such that for any $v \in V$, we have $\lambda v \in V$ for any $\lambda \in [0,1]$.
Therefore, $\{V,\, \R^n \setminus \Lambda^n_{sub}\}$ forms an open cover of $\R^n$.
Let $\{\mu,\nu\}$ be a smooth partition of unity subordinate to this covering.
Then $\hat{h}(x)=\mu(x)x$ is a smooth map $\hat{h}:\R^n \ra U$,
whose restriction to $\Lambda^n_{sub}$ is the identity map on $\Lambda^n_{sub}$.
Hence, $h=g \circ \hat{h}:\R^n \ra Y$ is the desired smooth extension of $f$.

Now let $A$ be a $D$-open subset of a fibrant diffeological space $X$.
Then for any smooth map $\alpha:\Lambda^n \ra A$,
we have a smooth extension $\beta:\R^n \ra X$ making the following diagram commutative:
\[
\xymatrix{\Lambda^n \ar[r]^\alpha \ar[d] & A\ \ar@{^{(}->}[r] & X. \\ \R^n \ar[urr]_\beta}
\]
Then $\beta^{-1}(A)$ is an open neighborhood of $\Lambda^n_{sub}$ in $\R^n$,
and $\alpha:\Lambda^n \ra A$ has a smooth lifting $\beta^{-1}(A) \ra A$.
Therefore, $\alpha$ has a smooth extension $\gamma:\R^n \ra A$ by the claim,
which implies the fibrancy of $A$.
\end{proof}

\begin{cor}\label{cpt}
Let $X$ be a diffeological space that is compact under the $D$-topology,
and let $N$ be a smooth manifold.
Then $C^\infty(X,N)$ is a fibrant diffeological space.
Moreover, if $x_0 \in X$ and $n_0 \in N$ are chosen points,
then the pointed mapping space $C^\infty((X,x_0),(N,n_0))$ is fibrant
when given the sub-diffeology.
\end{cor}

\begin{proof}
Let $N \ra \R^n$ be an embedding, and
let $U$ be an open tubular neighborhood of $N$ in $\R^n$,
so that the inclusion $i:N \ra U$ has a smooth retract $r:U \ra N$.
Then the composite
\[
\xymatrix{C^\infty(X,N) \ar[r]^{i_*} & C^\infty(X,U) \ar[r]^{r_*} & C^\infty(X,N)}
\]
is the identity map. That is, $C^\infty(X,N)$ is a smooth retract of $C^\infty(X,U)$.
To prove that $C^\infty(X,N)$ is fibrant, it is enough to prove that $C^\infty(X,U)$ is fibrant.
Since $X$ is compact, $C^\infty(X,U)$ is $D$-open in $C^\infty(X,\R^n)$; see~\cite[Proposition~4.2]{CSW}.
Note that $C^\infty(X,\R^n)$ is a diffeological group, hence fibrant.
By Theorem~\ref{Dopnfibrant}, $C^\infty(X,U)$ is fibrant.
The argument in the pointed case is similar.
\end{proof}

In particular, when $X$ is a point,
this corollary implies that every smooth manifold is fibrant.
This is the second proof of this fact.
Also, this corollary shows that the free loop space $C^\infty(S^1,N)$ 
and the based loop space $C^\infty((S^1,s_0),(N,n_0))$ 
of a smooth manifold are fibrant.

\begin{ex}
Let $X$ be a topological space.
Then $C(X)$ is a fibrant diffeological space,
since $D(\Lambda^n) \ra D(\R^n)$ has a retract in $\Top$.
This also follows from the proof of Proposition~\ref{3adjoint1}.
However, if $Y$ is a diffeological space,
then the natural map $Y \ra C(D(Y))$ is not always a weak equivalence in $\Diff$.
$Y=T^2_\theta$, the irrational torus of slope $\theta$, is such an example.
\end{ex}

Not every diffeological space is fibrant:

\begin{ex}\label{nonfib1}
$\Lambda^n$ is not fibrant for any $n \geq 2$,
since the natural injective map $\Lambda^n \ra \R^n$,
which is a trivial cofibration,
does not have a smooth retraction $\R^n \ra \Lambda^n$.
This follows immediately from the definition of the coequalizer diffeology on $\Lambda^n$.

Note that the inclusion map $\Lambda^n_{sub} \to \R^n$ is also a cofibration
and in fact has the left lifting property with respect to all fibrations.
This follows immediately from the fact that $\Lambda^n \to \Lambda^n_{sub} \to \R^n$
has this lifting property, where the first map is the natural bijection,
which is in particular an epimorphism.
Therefore, if $\Lambda^n_{sub}$ were fibrant,
then the inclusion map $i: \Lambda^n_{sub} \ra \R^n$ would have a smooth retraction $f: \R^n \ra \Lambda^n_{sub}$.
Suppose this is the case.
Then the composition $i \circ f: \R^n \ra \R^n$ is a smooth map preserving the axes,
and so $(i \circ f)_* = \id: T_0 \R^n \ra T_0 \R^n$.
This implies that $i \circ f$ is a local diffeomorphism at $0$ by the inverse function theorem,
which is a contradiction.

For the same reasons,
neither $\partial' \R^n$ nor $\partial' \R^n_{sub}$ is fibrant for any $n \geq 2$.

Similarly, many colimits of diffeological spaces are not fibrant.
For example, $\hat{S}^1$ defined in Example~\ref{hatS1} is not fibrant,
nor is the wedge of two or more smooth manifolds of positive dimension.
\end{ex}

\begin{ex}\label{nonfib3}
For any pointed diffeological space $(X,x)$,
we can construct the path space $P(X,x) = C^\oo((\R, 0), (X,x))$.
This diffeological space is always smoothly contractible, since we have a smooth
contracting homotopy
$\alpha:P(X,x) \times \R \ra P(X,x)$ defined by $\alpha(f,t)(s)=f(ts)$.
We also have a natural smooth map $\ev_1:P(X,x) \ra X$ defined by $f \mapsto f(1)$.
However, $\ev_1$ is not always a fibration in $\Diff$.
For example, take $X=\Lambda^n$ for $n \geq 2$, and let $x=0 \in X$.
It suffices to show that the fiber of $\ev_1$ at $x$, i.e., the loop space
$\Omega(X,x)$, is not fibrant.
We can construct a smooth map $H:X \ra \Omega(X,x)$ by
$H(y)(t)= \psi(t) y$,
where $\psi:\R \ra \R$ is a smooth function such that
$\psi(t)=0$ when $t \leq 0$ or $t \geq 1$ and $\psi(1/2)=1$.
Since $\ev_{1/2} \circ \, H = \id_{\Lambda^n}$,
$\Lambda^n$ is a smooth retract of $\Omega(X,x)$.
We saw in Example~\ref{nonfib1} that $\Lambda^n$ is not fibrant,
and so it follows that $\Omega(X,x)$ is not fibrant.

It is unfortunate that $\ev_1$ is not a fibration, as otherwise
a standard mapping path space construction could be used to factor
many maps into a weak equivalence followed by a fibration.
\end{ex}

\begin{ex}\label{nonfib2}
Write $\hat{\R}^n$ for $\R^n$ with the diffeology generated by a set
$S \subseteq C^\infty(\R^{n-1},\R^n)$ which contains all the natural inclusions
$\R^{n-1} \ra \R^n$ into the coordinate hyperplanes.
Then $\hat{\R}^n$ is not fibrant for $n \geq 1$.
If it were, then there would exist a smooth map $F : \R^n \to \hat{\R}^n$
making the diagram
\[
\xymatrix{\Lambda^n \ar[r]^i \ar[d]_i & \hat{\R}^n \ar[r]^j & \R^n \\ \R^n \ar@{.>}[ur]_F \\}
\]
commute, where $i$ is the usual inclusion and $j$ is the identity map.
Since $j \circ F$ is the identity map on the coordinate hyperplanes,
it must induce the identity map on the tangent spaces at $0$.
But, in a neighbourhood of $0$, the map $F$ must factor through $\R^{n-1}$,
which means that $j \circ F$ has rank at most $n-1$.

The same method shows that for any $n,m \in \N$ with $n>m$,
$\R^n$, with the diffeology generated by any set
$S \subseteq C^\infty(\R^m,\R^n)$ which contains all the natural inclusions $\R^m \ra \R^n$
into the coordinate $m$-planes, is not fibrant.
\end{ex}

\begin{ex}\label{halfline}
$X=[0,\infty)$ as a diffeological subspace of $\R$ is not fibrant.
It is not hard to show that $X \ra \R^0$ has the right lifting property with respect to $\Lambda^2 \ra \R^2$,
but we will see that it does not have the right lifting property with respect to $\Lambda^3 \to \R^3$.
Recall that $\Lambda^3$ is a colimit of three copies of $\R^2$ glued along three lines.
Let $f:\Lambda^3 \ra X$ be defined by $f_i:\R^2 \ra X$ with
$f_i(x_j,x_k)=(x_j-x_k)^2$ for $\{i,j,k\}=\{1,2,3\}$.
Assume that $f$ has a smooth extension $G:\R^3 \ra X$,
that is, that there exists a non-negative smooth function $F:\R^3 \ra \R$ such that
$F(x_1,x_2,0)=(x_1-x_2)^2$, $F(0,x_2,x_3)=(x_2-x_3)^2$ and $F(x_1,0,x_3)=(x_1-x_3)^2$.
Consider the composition $h:\xymatrix{\R \ar[r]^g & \R^3 \ar[r]^F & \R}$, with $g(t)=(t,t,t)$.
We compute some partial derivatives of $F$.
First, $F_1(x_1, x_2, 0) = \frac{\partial}{\partial x_1} (x_1 - x_2)^2 = 2(x_1 - x_2)$,
and so $F_{11}(x_1, x_2, 0) = 2$ and $F_{12}(x_1, x_2, 0) = -2$.
Thus $F_1(0,0,0) = 0$, $F_{11}(0, 0, 0) = 2$ and $F_{12}(0, 0, 0) = -2$.
Similarly, we find that $F_i(0,0,0) = 0$ and that $F_{ij}(0,0,0)$ is $2$ if $i = j$ and is $-2$ if $i \neq j$.
Clearly $h(0) = 0$.
By the chain rule, $h'(t) = \sum_i F_i(t,t,t)$, and so $h'(0) = 0$.
Also, $h''(t) = \sum_{i, j} F_{ij}(t,t,t)$, and so $h''(0) = -6$.
It follows that $h(t)=-3t^2+o(t^2)$,
which contradicts the fact that $F$ is non-negative.
\end{ex}

\begin{cor}\label{mfdwb-nonfib}
Any diffeological space containing $\R^{\geq 0} \times \R^n$ with the sub-diffeology of $\R^{n+1}$
as a $D$-open subset, for some $n \in \N$, is not fibrant.
In particular, any smooth manifold with boundary or with corners is not fibrant.
\end{cor}

\begin{proof}
This follows from Theorem~\ref{Dopnfibrant}, Example~\ref{halfline},
and the fact that the half line with the sub-diffeology of $\R$ is a smooth retract of $\R^{\geq 0} \times \R^n$.
\end{proof}

\begin{ex}\label{afewhalflines}\
\begin{enumerate}
\item For $n \geq 1$,
let $X_n$ be $[0,\infty)$ equipped with the diffeology
generated by the map $\R^n \ra [0,\infty)$ given by $x \mapsto \|x\|^2$.
Then the method used in Example~\ref{halfline} shows that $X_n$ is
not fibrant.
Note that $X_n$ is diffeomorphic to the quotient diffeological space $\R^n/O(n)$.

\item Note that $X_n \ra X_{n+1}$ given by $x \mapsto x$ is smooth.
Write $X_\infty$ for the colimit.
Then the diffeology on $X_\infty$ and the sub-diffeology from $\R$ are different;
see~\cite{IW} or~\cite[Remark~1.8.2]{W1}.
Moreover, the method used in Example~\ref{halfline} shows that $X_\infty$ is not fibrant.

\item Let $G$ be a finite cyclic group acting on $\R^2$ by rotation.
By a similar method, one can show that the orbit space $\R^2/G$ with the quotient diffeology is not fibrant.

\item Let $\Z_2$ act on $\R^n$ by $(x_1,\ldots,x_n) \mapsto (\pm x_1,\ldots,\pm x_n)$.
Then the orbit space $\R^n / \Z_2$ with the quotient diffeology is not fibrant,
since $X_1$ is a smooth retract of $\R^n / \Z_2$.
\end{enumerate}
\end{ex}

\end{document}